\documentclass[]{article}

\usepackage{todonotes}

\usepackage{graphicx}
\usepackage{amsfonts}
\usepackage{amsmath}
\usepackage{amssymb}
\usepackage{fancyhdr}
\usepackage{titlesec}
\usepackage{indentfirst}
\usepackage{booktabs}
\usepackage{verbatim}
\usepackage{color}
\usepackage{amsthm}
\usepackage{url}
\usepackage{cite}
\usepackage{epstopdf}
\usepackage[page,toc,titletoc]{appendix}
\usepackage{titletoc}
\usepackage[colorlinks=true]{hyperref}

\numberwithin{equation}{section}
\newtheorem{theorem}{Theorem}[section]
\newtheorem{lemma}[theorem]{Lemma}
\newtheorem{corollary}[theorem]{Corollary}
\newtheorem{proposition}[theorem]{Proposition}
\newtheorem{definition}[theorem]{Definition}
\newtheorem{remark}[theorem]{Remark}

\newcommand{\rd}{\mathrm{d}}

\newcommand{\bR}{\mathbb{R}}

\newcommand{\Sd}{\mathbb{S}}

\newcommand{\vs}{v_{*}}

\newcommand{\am}{a_{-}}
\newcommand{\ap}{a_{+}}
\newcommand{\xie}{\xi_{e}}
\newcommand{\p}{\varphi}
\newcommand{\K}{\mathcal{K}}

\newcommand*{\im}{\mathop{}\!\mathrm{i}}
\newcommand*{\e}{\mathop{}\!\mathrm{e}}

\begin{document}

\title{On the Measure Valued Solution to the Inelastic Boltzmann Equation with Soft Potentials}

\author{Kunlun Qi\footnote{Department of Mathematics, City University of Hong Kong, Hong Kong, People's Republic of China (kunlun.qi@my.cityu.edu.hk).}}   

\maketitle


\begin{abstract}
	The goal of this paper is to extend the existence result of measure-valued solution to the Boltzmann equation in elastic interaction, given by Morimoto-Wang-Yang, to the inelastic Boltzmann equation with moderately soft potentials, also as an extensive work of our preceding result in the inelastic Maxwellian molecules case.
	We prove the existence and uniqueness of measure-valued solution under Grad's angular cutoff assumption, as well as the existence of non-cutoff solution, for both finite and infinite energy initial datum, by a delicate compactness argument. 
	In addition, the moments propagation and energy dissipation properties are justified for the obtained measure-valued solution as well.
\end{abstract} 

{\small 
	
	
}

\tableofcontents

\section{Introduction}
\label{sec:intro}

Consider the inelastic spatially homogeneous Boltzmann equation in $ \mathbb{R}^{3} $,
\begin{equation}\label{IB}
\partial_{t} f(t,v) = Q_{e}(f,f) (t,v)
\end{equation}
associated with the non-negative initial condition
\begin{equation}\label{F0}
f(0,v) = F_{0},
\end{equation}
where the unknown density function $ f=f(v,t) $ is regarded as the density function of a probability distribution, or more generally, a probability measure; and the initial datum $ F_{0} $ is also assumed to be a non-negative probability measure on $ \mathbb{R}^{3} $. In this case, the inelastic Boltzmann collision operator can be written as
\begin{equation}
Q_{e}(g,f) = \int_{\bR^{3}} \int_{\Sd^{2}} B\left(|v-v_{*}|, \frac{v-v_{*}}{|v-v_{*}|} \cdot \sigma\right) \left[ J g('\vs)f('v)- g(\vs)f(v) \right] \,\rd \sigma \,\rd \vs,
\end{equation}
where the relations between the pre-collisional velocities $ 'v, '\vs $ and the corresponded post-collisional velocities $ v,\vs $ are given as follows:
\begin{equation}
\left\{
\begin{array}{lr}
'v = \frac{v+\vs}{2} - \frac{1-e}{4e}(v-\vs) + \frac{1+e}{4e}|v-\vs|\sigma  &  \\
'\vs = \frac{v+\vs}{2} + \frac{1-e}{4e}(v-\vs) - \frac{1+e}{4e}|v-\vs|\sigma, &  
\end{array}
\right.
\end{equation}
note that $ e\in (0,1] $ is the so-called restitution coefficient, describing the inelastic affects in the collision between particles ($ e=1 $ denotes elastic collision) as in \cite{Villani2006granularmaterials}, which is assumed to be a constant in this paper; then, since the following discussion will be based on the weak formulation of the collision operator, here we generally denote $ J $ as the Jacobian of the transform from $ (v,\vs) $ to $ ('v ,'\vs) $, and one can refer to \cite{WeiZhang2012} for more detailed representation forms in non-constant restitution coefficient case.


\textit{1.1 Collision kernel.}

Here we consider the non-negative collision kernel $ B\left(|v-v_{*}|, \frac{v-v_{*}}{|v-v_{*}|}\cdot \sigma\right) $, for any unit vector $ \sigma\in\mathbb{S}^{2} $, which implies that $ B $ does not only depend on the deviation angle, but also the relative velocity, 
\begin{equation}\label{Bb}
B\left(|v-v_{*}|, \frac{v-v_{*}}{|v-v_{*}|}\cdot \sigma\right) = \Phi\left(|v-v_{*}|\right) b\left( \cos\theta \right), \quad \cos\theta =\frac{v-v_{*}}{|v-v_{*}|}\cdot \sigma, 
\end{equation}
where $ \Phi\left(|v-v_{*}|\right) = \left| v-v_{*} \right|^{\gamma} $ for some $ \gamma > -3 $. As in the elastic case, the $ \gamma > 0 $ case is called the hard potential, the $ -2 <  \gamma < 0 $ case is called the moderately soft potential, the $ -3 < \gamma < -2 $ case is called strongly soft potential, and the case $ \gamma = 0 $ refers to the Maxwellian molecule.\\
In addition, the angular collision part $b(\cos\theta)$ is an implicitly defined function, asymptotically behaving as, when $\theta\rightarrow 0^{+}$,
\begin{equation}\label{noncutoffnu}
\sin \theta b(\cos\theta) \big|_{\theta\rightarrow 0^{+}} \sim K\theta^{-1-2s}, \quad \text{for}\ 0<s<1 \ \text{and} \  K >0,
\end{equation}
i.e., it has a \textit{non-integrable singularity} when the deviation angle $ \theta $ is small. 
The range of deviation angle $ \theta $, namely the angle between pre- and post-collisional velocities, is a full interval $ \left[0,\pi\right] $, but it is customary to restrict it to $ \left[0,\pi/2\right] $ mathematically, replacing $ b(\cos\theta) $ by its ``symmetrized" version \cite{morimoto2016measure}:
\begin{equation}
\left[ b(\cos\theta) + b(\cos\left(\pi-\theta\right))\right]\mathbf{1}_{0\leq \theta \leq \frac{\pi}{2}}.
\end{equation}
As it has been long known, the main difficulty in establishing the well-posedness result for Boltzmann equation is that the singularity of the collision kernel $ b $ is not locally integrable in $ \sigma\in\Sd^{2} $. To avoid this, Harold Grad proposed the integrable assumption \cite{GradCutoff} on the collision kernel $ b $ by a ``cutoff " near singularity; however, here the full singularity condition for the collision kernel with \textit{non-cutoff assumption} is considered
\begin{equation}\label{noncutoffb}
\exists \alpha_{0} \in (0,2], \quad \text{such that} \quad \int_{0}^{\frac{\pi}{2}} \sin^{\alpha_{0}}\left(\frac{\theta}{2}\right) b(\cos\theta) \sin\theta \,\rd\theta < \infty,
\end{equation}
which can handle the strongly singular kernel $ b $ in \eqref{noncutoffnu} with some $ 0 < s < 1 $ and $ \alpha_{0}\in(2s, 2 ] $.

\textit{1.2 Weak formulation of the inelastic operator.}

%

For many purposes, it will be more convenient to work with the weak formulation\cite{GambaDiffusively}, which is quite natural form from the physical view of point like the Maxwellian molecule case in our precedent work \cite{Qi20} that: letting $ \psi(v) $ be a test function, and then by changing of variable and the symmetry argument, we have
\begin{equation}\label{weak}
\begin{split}
&\int_{\bR^{d}} Q_{e}(f,f)(v)\psi(v) \rd v\\
=& \int_{\bR^{3}} \int_{\bR^{3}} \int_{\Sd^{2}} b\left(\frac{v-\vs}{|v-\vs|}\cdot \sigma\right)\Phi \left(|v-v_{*}|\right) f(\vs) f(v) \left[ \psi(v')  - \psi(v) \right] \,\rd \sigma \,\rd \vs \,\rd v\\
=& \frac{1}{2}\int_{\bR^{3}} \int_{\bR^{3}} \int_{\Sd^{2}} b\left(\frac{v-\vs}{|v-\vs|}\cdot \sigma\right) \Phi\left(|v-v_{*}|\right) f(\vs) f(v) \left[ \psi(v') + \psi(\vs') - \psi(v) - \psi(\vs) \right] \,\rd \sigma \,\rd \vs \,\rd v
\end{split}
\end{equation}
such that the particular form of the inelastic collision law is only manifested in the test function $ \psi(v') $ and $ \psi(\vs') $, where the post-collisional velocities $ v', v'_{*} $ (with $ v,\vs $ taken as the pre-collisional velocities) are 
\begin{equation}\label{ve}
\left\{
\begin{array}{lr}
v' =  \frac{v+\vs}{2} + \frac{1-e}{4}(v-\vs) + \frac{1+e}{4}|v-\vs|\sigma &\\
v'_{*} = \frac{v+\vs}{2} - \frac{1-e}{4}(v-\vs) - \frac{1+e}{4}|v-\vs|\sigma.&
\end{array}
\right.
\end{equation}

\textit{1.3 Conservative and dissipative law.}

Furthermore, we also introduce another type of representation for the post-collisional velocities $ v'$ and $\vs' $, that is called the $ \omega $-form, 
\begin{equation}
\left\{
\begin{array}{lr}
v' = v - \frac{1+e}{2} \left[ (v-\vs)\cdot \omega \right]\omega &  \\
\vs' = \vs + \frac{1+e}{2} \left[ (v-\vs)\cdot \omega \right]\omega, &  
\end{array}
\right.
\end{equation}
from which, we can easily verify the conservation of momentum and dissipation of energy:
\begin{equation}
v+\vs = v'+\vs'; \quad |v'|^{2} + |\vs'|^{2} - |v|^{2} - |\vs|^{2} = -\frac{1-e^{2}}{2} \left[ (v-\vs)\cdot \omega \right] \leq 0.
\end{equation}
Note that in this case $ '\vs, 'v $ do not coincide with $ \vs', v' $ since the inelastic collisions are not revertible. Also note that,
\begin{equation}\label{QEcon}
\int_{\bR^{d}} Q_{e}(f,f)(v) \,\rd v = 0, \quad \int_{\bR^{d}} Q_{e}(f,f)(v) v \,\rd v = 0
\end{equation}
but
\begin{equation}\label{QEdiss}
\int_{\bR^{d}} Q_{e}(f,f)(v) |v|^{2} \,\rd v \leq 0.
\end{equation}

\section{Main Results}

\subsection{Motivation}
In recent years, the inelastic Boltzmann equation has been widely applied in the study of evolutionary phenomenon for granular gases, where the interactions are described by inelastic collisions; however, the mathematical kinetic theory of granular gases, especially for the general and physical non-cutoff model, is still restrictive. In the framework under cutoff assumption, Bobylev-Carrillo-Gamba have initially established the well-posed theory of the three dimensional inelastic Boltzmann equation with Maxwellian molecules in \cite{Bobylev2000inelastic}, and Mischler-Mouhot have further systematically study the inelastic hard sphere model in their series of work \cite{MM2006hardsphere1,MM2006hardsphere2}. Besides, to overcome the internally diffusive-driven property of the inelastic Boltzmann collision operator, a heat bath term is frequently added to keep the system out of the ``freezing" state, and there are lots of work devoted to this kind of model, see \cite{GambaDiffusively,BCT2006,MengLiu2020} and the references therein. More recently, not only for the constant restitution coefficients, some results of inelastic Boltzmann equation with non-constant restitution coefficient have been extended by Alonso and collaborators in \cite{Alonso2009, AL2010, AL2014CMP, ABCL2018}. And we refer to a recent textbook \cite{Brilliantov2004} by Brilliantov-P{\"o}schel for the developments from physical view of point, as well as a review paper \cite{Villani2006granularmaterials} by Villani for more mathematical progress in this field.

On the other hand, the theory about finding the measure-valued solution to the elastic Boltzmann equation has been well-developed in the past decade; since $ f $ in \eqref{IB} itself is density function, it is pretty natural to consider the measure-valued solution if the initial datum $ F_{0} $ is a probability measure as well. Initiated from the probability metric proposed by Toscani-Villani in \cite{ToscaniVillani1999}, Cannone-Karch in \cite{cannone2010infinite} further made pioneered contribution to obtain the well-posedness in the probability measure space of elastic Boltzmann equation with Maxwellian molecules, which doesn't exclude infinite energy case but only handles the mild singularity of non-cutoff collision kernel. Fortunately, Morimoto extended their results to the strong singularity as well as proving some smoothing effect in \cite{morimoto2012remark}. Besides, Lu-Mouhot showed existence of weak measure-valued solution without cutoff assumption for hard potential, having finite mass and energy, in \cite{LuMouhot2012,LuMouhot2015}. Moreover, Morimoto-Wang-Yang further studied the measure-valued solution in more general non-cutoff case (including finite energy and infinite energy initial datum, hard potential and soft potential), as well as the 
moments and smoothing  property in their series paper \cite{morimoto2016measure, MWY2015,CMWY2016}. 

Hence, considering the classical results above, our contribution in this paper is to establish the systematic existence theory of the measure-valued solution to the inelastic Boltzmann equation with soft potentials, where very few results are available to our best knowledge except for some relevant external-forced model \cite{WeiZhang2012} or Vlasov-Poisson-Boltzmann System \cite{CH2014}, and \cite{ETB2006} is referred for more detailed physical motivation for this model. In fact, we have already studied the Maxwellian molecules case in our preceding work \cite{Qi20}, based on which, here for more general collision kernel with soft potential, the existence and uniqueness of measure-valued solution under cutoff assumption are proved by Banach Contraction Theorem, and the existence of non-cutoff solution, in cases of both finite and infinite energy initial datum, is further established by a careful compactness argument. It is worth noting that, in contrast with the Maxwellian molecules case, the main difficulties lie in dealing with the soft potential collision kernel, which has singularity nearby the origin point; to achieve this, another cutoff function of soft potential that has worked for the elastic case \cite{morimoto2016measure} is introduced again, and will play a key role in handling the trouble-maker. What's more, the moments propagation property is shown by taking the similar strategy as the elastic case; and the energy dissipation feature, which is distinct for the inelastic model, is also justified by the delicate analysis. At last, the additional reason that motivates us to study the existence theory for the inelastic model is our recent progress \cite{HM19, HQ2020, HQY20} in the numerical approximation of Boltzmann collision operator; having the existence results in hand, we can further logically and heuristically apply them in the numerical simulation as well as the stability analysis of corresponded numerical solution.

\subsection{Main Theorems}
In this paper, we still try to look for the solution which is a probability measure $ F_{t} $ for any $ t \geq 0 $, \textit{i.e.}, $ F_{t} \in C\left([0,\infty), P_{\alpha}(\mathbb{R}^{3})  \right) $, where $ P_{\alpha}(\mathbb{R}^{3}) $ is the set of probability measures on $ \mathbb{R}^{3} $ with finite moments up to the order $ \alpha\in[0,2] $, which implies the possible existence of infinite energy solution, more precisely,
\begin{equation}\label{Palpha}
\begin{split}
P_{\alpha}(\mathbb{R}^{3}) = \{ F \in P_{0}(\mathbb{R}^{3})& \big|  \int_{\mathbb{R}^{3}} \,\rd F(v) =1, \ \int_{\mathbb{R}^{3}} |v|^{\alpha} \,\rd F(v) < \infty\\
&\text{and if}\ \alpha > 1, \int_{\mathbb{R}^{3}} v_{j} \,\rd F(v) = 0, \ j = 1,2,3 \}
\end{split}
\end{equation}
and 
the continuity of the map $ t\in [0,\infty)  \rightarrow F_{t} \in P_{\alpha}(\mathbb{R}^{3}) $ is in the weak sense that 
\begin{equation}
\lim\limits_{t \rightarrow t_{0}} \int_{\mathbb{R}^{3}} \psi \,\rd F_{t}(v) = \int_{\mathbb{R}^{3}} \psi \,\rd F_{t_{0}}(v), \quad \forall \psi \in C_{b}(\mathbb{R}^{3}),
\end{equation}
note that the space $ C_{b}(\mathbb{R}^{3}) $ includes all the continuous with certain type of decay condition at the infinity of $ v\in \mathbb{R}^{3} $, defined as following,
\begin{equation}
C_{b}(\mathbb{R}^{3}) : = \left\lbrace \psi \in C(\mathbb{R}^{3}); \quad  \sup_{v \in\mathbb{R}^{3}} \frac{\left| \psi(v)\right|}{\left\langle v \right\rangle^{\alpha} } < \infty, \quad \left\langle v \right\rangle = \sqrt{1+ |v|^{2}} \right\rbrace.
\end{equation}

In this paper, the main result is to show that the Cauchy problem of the inelastic Boltzmann equation admits a measure-valued solution in the case of moderately soft potential ($ -2 \leq \gamma < 0 $) for both finite and infinite energy initial datum. 

For the completeness, we first introduce the specific definition of measure-valued solution to the inelastic Boltzmann equation following the elastic definition in \cite{morimoto2016measure}.
\begin{definition}\label{measure}
	(Measure-valued solution to inelastic Boltzmann equation) Let $ e \in (0,1] $ and the collision kernel $ B $ satisfy the \eqref{Bb}-\eqref{noncutoffb}. For any $ F_{0} \in P_{\alpha}(\mathbb{R}^{3}) $ with $ 0 < \alpha \leq 2 $. We define $ F_{t} \in C\left([0,\infty), P_{\alpha}(\mathbb{R}^{3})  \right)  $ as a measure-valued solution to the Cauchy problem \eqref{IB}-\eqref{F0} if it satisfies:\\
	(1) For every $ \psi(v) \in C_{b}^{2}(\mathbb{R}^{3}) $ and $ t > 0 $,
	\begin{equation}
	\int_{0}^{t} \int_{\bR^{3}} \int_{\bR^{3}} \int_{\Sd^{2}} b\left(\frac{v-\vs}{|v-\vs|}\cdot \sigma\right) \left|v-v_{*}\right|^{\gamma} \left| \psi(v'_{*}) + \psi(v') - \psi(v_{*}) - \psi(v) \right| \,\rd \sigma \,\rd F_{\tau}(v) \,\rd F_{\tau}(v_{*}) \,\rd \tau 
	\end{equation}
	is finite.\\
	(2) For every $ \psi(v) \in C_{b}^{2}(\mathbb{R}^{3}) $,
	\begin{equation}\label{weakmeasure}
	\begin{split}
	\int_{\bR^{3}} \psi(v) \,\rd F_{t}(v) = &\int_{\bR^{3}} \psi(v) \,\rd F_{0} + \frac{1}{2} \int_{0}^{t} \int_{\bR^{3}} \int_{\bR^{3}} \int_{\Sd^{2}} b\left(\frac{v-\vs}{|v-\vs|}\cdot \sigma\right) \left|v-v_{*}\right|^{\gamma} [ \psi(v'_{*}) + \psi(v')\\
	& - \psi(v_{*}) - \psi(v) ] \,\rd \sigma \,\rd F_{\tau}(v) \,\rd F_{\tau}(v_{*}) \,\rd \tau 
	\end{split}
	\end{equation}
	(3) If $ \alpha \geq 1 $, then the momentum conservation law holds:
	\begin{equation}\label{momencon}
	\forall t \geq 0, \quad \int_{\bR^{3}} v_{j} \,\rd F_{t}(v) = \int_{\bR^{3}} v_{j}\,\rd F_{0}(v), \quad j=1,2,3.
	\end{equation}
	(4) If $ \alpha =2  $, then $ F_{t} \in C\left([0,\infty), P_{\alpha}(\mathbb{R}^{3})  \right)  $ and then energy dissipation law holds:
	\begin{equation}\label{enerdiss}
	\forall t \geq 0, \quad \int_{\bR^{3}} |v|^{2} \,\rd F_{t}(v) \leq \int_{\bR^{3}} |v|^{2} \,\rd F_{0}(v).
	\end{equation}
\end{definition}

\begin{remark}
	(i) Note that, compared with the elastic case, the post-collisional velocities $ v',v'_{*} $ above are described by \eqref{ve} in the inelastic collision process; (ii) The definition above does not require finite entropy condition, which is consistent with \cite[Definition 1.1]{morimoto2016measure}.
\end{remark}

As usual, we also defined the similar operator $ L^{e}_{B}[\psi](v,v_{*}) $ that, for any $ \psi\in C_{b}^{2}(\mathbb{R}^{3}) $,
\begin{equation}
L^{e}_{B}[\psi](v,v_{*}) := \int_{\Sd^{2}} b\left(\frac{v-\vs}{|v-\vs|}\cdot \sigma\right) \left|v-v_{*}\right|^{\gamma} \left[ \psi(v'_{*}) + \psi(v') - \psi(v_{*}) - \psi(v) \right] \,\rd \sigma.
\end{equation}

Now we're ready to state our main theorem:
\begin{theorem}\label{main1}
	Let $ e\in(0,1] $ and the collision kernel $ B $ satisfy the \eqref{Bb}-\eqref{noncutoffb} with $ -2 \leq  \gamma < 0 $. 
	
	(i) (Finite energy initial datum) For any initial datum $ F_{0} \in P_{2}\left(\mathbb{R}^{3}\right) $, then there exists a measure-valued solution $ F_{t} \in C\left(  \left[0,\infty\right), P_{2}\left( \mathbb{R}^{3} \right) \right) $ to problem \eqref{IB}-\eqref{F0}.
	
	(ii) (Infinite energy initial datum) For any initial datum $ F_{0} \in P_{\alpha}\left(\mathbb{R}^{3}\right) $ with $ \alpha \in \left( c_{\gamma,s} ,2 \right) $, where 
	\begin{equation}
	c_{\gamma,s} =
	\begin{cases}
	\max\{ \frac{\gamma}{2s}+ 1 , 0 \}   & \text{if} \ 0 < s < \frac{1}{2}, \\
	\max\{\gamma + 2s\} &  \text{if} \ \gamma+2s < 1 \ \text{and} \ \frac{1}{2} \leq s < 1, \\
	\frac{\gamma}{2s-1} + 2    & \text{if} \ \gamma+2s \geq 1.
	\end{cases}
	\end{equation}
	Then there exists a measure-valued solution $ F_{t} \in C\left(  \left[0,\infty\right), P_{\alpha}\left(\mathbb{R}^{3}\right) \right) $ to problem \eqref{IB}-\eqref{F0}.
\end{theorem}

Meanwhile, we can further obtain the moment propagation property of any measure-valued solution $ F_{t} \in C \left( [0,\infty), P_{2}\left(\mathbb{R}^{3}\right)  \right) $ in the case of finite energy initial datum.
\begin{corollary}\label{momentspro}
	Let $ e \in (0,1] $ and the collision kernel $ B $ satisfy the \eqref{Bb}-\eqref{noncutoffb} with $ -2 \leq  \gamma < 0 $. For any measure-valued solution $ F_{t} \in C\left(  \left(0,\infty\right), P_{2}\left(\mathbb{R}^{3}\right) \right) $ to problem \eqref{IB}-\eqref{F0} with the initial datum $ F_{0} \in P_{2}(\mathbb{R}^{3}) $, the moment propagation property holds in the sense that:
	for any $ l > 0 $ and
	\begin{equation}
	\int_{\bR^{3}} |v|^{l} \,\rd F_{0}(v) < \infty,
	\end{equation}
	then for all $ T \in \left( 0,\infty \right) $,
	\begin{equation}
       \sup\limits_{t\in [0,T]} \int_{\bR^{3}} |v|^{l} \,\rd F_{t}(v) < \infty.
       \end{equation}   
\end{corollary}

\subsection{Plan of Paper}
The paper is organized as follows. In the next section \ref{sec:pre} we will first introduce some preliminary properties of the characteristic function $ \p $ as well as the cutoff function $ \Phi_{c} $, which have the key merits in constructing the measure-valued solution and handling the singularity of the soft potential collision kernels. In section \ref{sec:cutoff}, the well-posed theory under cutoff assumption is proved by using the Banach Contraction Theorem, meanwhile, the conservative and dissipative properties of the various moments of the corresponded cutoff solution are shown as well. Finally, based on the cutoff solutions, the existence results without the cutoff assumption are established in section \ref{sec:noncutoff} by compactness argument for finite energy and infinite energy initial datum, respectively.

\section{Preliminaries}\label{sec:pre}

\subsection{Some Properties of Characteristic Functions}
\label{subsec:characteristic}

To illustrate our working function space, the characteristic function space is first introduced by following the same definition as in \cite{cannone2010infinite}, which consists of the Fourier transformations for the probability measures.

\begin{definition}
	A function $ \p:= \mathbb{R}^{d} \mapsto \mathbb{C} $ is called a characteristic function if there is a probability measure $ \mu $ (i.e. a Borel measure with $ \int_{\mathbb{R}^{d}} \mu \,\rd v = 1 $) such that the identity holds: $ \p(\xi) = \hat{\mu}(\xi) = \int_{\mathbb{R}^{d}} \e^{-\im v\cdot \xi} \mu \,\rd v $. Then the set of all characteristic function $ \p := \mathbb{R}^{d} \mapsto \mathbb{C} $ is denoted by $\mathcal{K}$.
\end{definition}

For the self-contained purposes, we include the following celebrated Theorem \ref{character1} by Bochner, which introduces an equivalent condition of the characteristic function:
\begin{theorem}\label{character1}
	A function $ \p : \mathbb{R}^{3} \rightarrow \mathbb{C} $ is a characteristic function if and only if the following conditions hold:\\
	(i) $ \p $ is a continuous function on $ \mathbb{R}^{3} $.\\
	(ii) $ \p(0) = 1 $.\\
	(iii) $ \p $ is positive-definite, \textit{i.e.}, for every $ k\in \mathbb{N} $, every vector $ \xi^{1},...,\xi^{k} \in \mathbb{R}^{3} $ and all $ \lambda_{1},...,\lambda_{k} \in \mathbb{C} $, such that
	\begin{equation}
	\sum_{j,l}^{k} \p\left(\xi^{j} - \xi^{l} \right) \lambda_{j} \overline{\lambda_{l}} \geq 0.
	\end{equation}
\end{theorem}
Then the following Lemma \ref{character2} can be obtained immediately from the definition of positive-definite functions:
\begin{lemma}\label{character2}
	Any linear combination with positive coefficients of positive-definite functions is a positive-definite function. The set of positive-definite functions is closed with respect to pointwise convergence.
\end{lemma}
For the proof of the classical results above and more general properties about the positive-definite functions, we refer to the \cite[Sect.3]{cannone2010infinite}, \cite[Sect.2.1]{morimoto2016measure}, and the references therein.

In addition, we also introduce the space $ \K^{\alpha} $, initially defined by Cannone-Karch in \cite{cannone2010infinite}, which is the subspace of $ \K $ as following:
\begin{equation}
\K^{\alpha} = \left\lbrace \p\in\K; \ \left\| \p-1 \right\|_{\alpha}<\infty \right\rbrace 
\end{equation}
where 
\begin{equation}
\left\| \p-1 \right\|_{\alpha} = \sup_{\xi\in\mathbb{R}^{3}} \frac{\left| \p(\xi)-1 \right|}{|\xi|^{\alpha}} .
\end{equation}
Meanwhile, according to \cite[Proposition 3.10]{cannone2010infinite}, the space $ \mathcal{K}^{\alpha} $ equipped with the distance 
\begin{equation}
\|\p - \tilde{\p} \|_{\alpha} = \sup\limits_{\xi \in\mathbb{R}^{3}} \frac{|\p(\xi) - \tilde{\p}(\xi)|}{|\xi|^{\alpha}}
\end{equation}
is a complete metric space; and noticing that $ \mathcal{K}^{\alpha} = \{1\} $ for all $ \alpha > 2 $, the following embedding relation holds
\begin{equation}\label{embed}
\{1\} \subset \mathcal{K}^{\alpha} \subsetneq \mathcal{F}(P_{\alpha_{0}}) \subsetneq \mathcal{K}^{\alpha_{0}} \subset \mathcal{K}^{0} = \mathcal{K}
\end{equation}
for all $ 2 \geq \alpha \geq \alpha_{0} \geq 0 $ if $ \alpha \neq 1 $.
More general properties and estimates of the functions in $ \mathcal{K} $ and $ \mathcal{K}^{\alpha} $ are summarized in our preceding work \cite{Qi20}. Here we only recall the following technical Lemma \ref{L1}, which is useful for our proof in the rest of the paper.
\begin{lemma}\label{L1}
For any positive-definite function $ \p=\p(\xi)\in\mathcal{K} $ such that $ \p(0) = 1 $, we have
\begin{equation}\label{l1}
\left|\p(\xi) - \p(\eta)\right|^{2} \leq 2\left( 1 - \text{Re} \left[ \p(\xi-\eta) \right]\right) 
\end{equation}
and
\begin{equation}\label{l2}
\left|\p(\xi)\p(\eta) - \p(\xi+\eta) \right|^{2} \leq \left(1-\left|\p(\xi) \right|^{2 }\right)\left(1-\left|\p(\eta) \right|^{2 }\right)
\end{equation}
Moreover, if $ \p=\p(\xi)\in\mathcal{K}^{\alpha} $, then we have
\begin{equation}\label{l3}
\left| \p(\xi) - \p(\xi+\eta) \right| \leq \left\| \p-1 \right\|_{\alpha}\left( 4|\xi|^{\frac{\alpha}{2}} |\eta|^{\frac{\alpha}{2}} + |\eta|^{\alpha} \right).
\end{equation}
\end{lemma}
\begin{proof}
	The proof of \eqref{l1}-\eqref{l2} is based on the definition of positive-definite function, which is proved in detail in \cite[Lemma 3.8]{cannone2010infinite}. And the proof of \eqref{l3} can be found in \cite[Lemma 2.1]{morimoto2012remark}. 
\end{proof}

\subsection{Some Properties of Cutoff Function $\Phi_{c}$. }
In this subsection, we first introduce the cutoff function $ \Phi_{c} $ for the kinetic collision kernel $ \Phi(|v-v_{*}|) = |v-v_{*}|^{\gamma} $.
\begin{definition}\label{D1}
For the moderately soft potential case $ \left( -2 \leq \gamma < 0\right) $, $ \Phi_{c} $ takes the form $ \Phi_{c}\left( |v-v_{*}| \right) = |v-v_{*}|^{\gamma}\phi_{c}\left( |v-v_{*}| \right) $, where the smooth function $ \phi_{c}\left(|v-v_{*}|\right) \in\mathcal{C}_{c}^{\infty}\left(\mathbb{R}^{3}\right) $ satisfies:\\
(i) $ \phi_{c} $ is supported on $ \left\lbrace v-v_{*}: 1/2r \leq |v-v_{*}| \leq 2r \right\rbrace $ with $ 0\leq \phi_{c} \leq 1  $;\\
(ii) $ \phi_{c} = 1 $ on $ \left\lbrace v-v_{*}: 1/r \leq |v-v_{*}| \leq r \right\rbrace $.
\end{definition}
Then, the useful estimate of the function $ \hat{\Phi}_{c} $ as in \cite{morimoto2016measure} is also presented as following,
\begin{lemma}\label{PhiDecay}
	Let the cutoff kinetic collision kernel $ \Phi_{c} $ be defined as in Definition \ref{D1}. 
       For the soft potential case $ \gamma < 0 $, we have, for all $ k,N\in\mathbb{N} $ and $ \zeta\in\mathbb{R}^{3} $, there exists a constant $ C_{r,N,k} $ such that
	\begin{equation}\label{PhiDecaySoft}
	\left| \partial_{\zeta}^{k} \hat{\Phi}_{c}(\zeta) \right| \lesssim \frac{C_{r,N,k}}{\left< \zeta\right>^{2N}}, \quad  \text{if}\  N \geq 1 \ \text{and}\  k \geq 0,
	\end{equation}
	where $ \hat{\Phi}_{c} $ is Fourier transformation of $ \Phi_{c} $ defined as in the following formula \eqref{IBE}.
\end{lemma}
\begin{proof}
	The proof has been given in \cite[Lemma 2.5]{morimoto2016measure}.
\end{proof}

\section{Existence and Uniqueness with Cutoff Assumption}\label{sec:cutoff}

In this section, we consider the Cauchy problem with cutoff collision kernel in the sense that
\begin{equation}\label{IBEcut}
\partial_{t} f(t, v) = \int_{\bR^{3}} \int_{\Sd^{2}} b_{c}\left(\frac{v-v_{*}}{|v-v_{*}|} \cdot \sigma\right) \Phi_{c}(|v-v_{*}|)  \left[ J f(\vs')f(v')- f(\vs)f(v) \right] \,\rd \sigma \,\rd \vs,
\end{equation}
associated with the non-negative initial condition, including both finite and infinite energy case,
\begin{equation}\label{F0cut}
f(0,v) = F_{0}  \in P_{\alpha_{0}}(\mathbb{R}^{3}), \quad 0 < \alpha_{0} \leq 2,
\end{equation}
where, without loss of generality, the angular part $ b_{c} $ is assumed to be integrable and normalizable such that
\begin{equation}\label{cutoff}
\int_{\Sd^{2}} b_{c}\left(\frac{v-v_{*}}{|v-v_{*}|} \cdot \sigma\right) \,\rd \sigma = 1
\end{equation}
and the function $ \Phi_{c} $, the cutoff version of the kinetic part of collision kernel $ \Phi\left( v-v_{*} \right) =  \left| v-v_{*} \right|^{\gamma} $, is defined as in Definition \ref{D1}. 

\begin{theorem}\label{local}
	\emph{(Well-posedness under cutoff assumption)} For any $ e\in(0,1] $ and $ \alpha_{0}\in \left(0,2\right] $. Let the collision kernel $b_{c}$ satisfy the cutoff assumption \eqref{cutoff} and $ \Phi_{c} $ satisfy the Definition \ref{D1}. Then, for any initial datum $ F_{0}\in P_{\alpha_{0}}(\mathbb{R}^{3}) $, there exists a unique solution $ F_{t} \in C\left( \left[0,\infty \right), P_{\alpha} \right) $ with $ \alpha \in (0,\alpha_{0}) $ to problem \eqref{IBEcut}-\eqref{F0cut} in the measure-valued sense that, for $ \psi(v) \in C_{b}^{2}(\mathbb{R}^{3}) $,
	\begin{equation}
	\begin{split}
	\int_{\bR^{3}} \psi(v) \,\rd F_{t}(v) = &\int_{\bR^{3}} \psi(v) \,\rd F_{0} + \frac{1}{2} \int_{0}^{t} \int_{\bR^{3}} \int_{\bR^{3}} \int_{\Sd^{2}} b_{c}\left(\frac{v-v_{*}}{|v-v_{*}|} \cdot \sigma\right) \Phi_{c}(|v-v_{*}|)\\ &[ \psi(v'_{*}) + \psi(v')
	 - \psi(v_{*}) - \psi(v) ] \,\rd \sigma \,\rd F_{\tau}(v) \,\rd F_{\tau}(v_{*}) \,\rd \tau.
	\end{split}
	\end{equation}
\end{theorem}

The main strategy to prove the Theorem \ref{local} is transforming the study in the original physical space to the study in characteristic space. Indeed, the inelastic Boltzmann equation with general collision kernels after Fourier transformation can be written as following, see Appendix \ref{appen1} for complete derivation,
\begin{equation}\label{IBE}
\partial_{t} \varphi(\xi,t) = \int_{\Sd^{2}} b_{c}\left(\frac{\xi\cdot\sigma}{|\xi|}\right) \int_{\bR^{3}} \hat{\Phi}_{c}(\zeta) \left[ \varphi(\xie^{+}-\zeta, t)\varphi(\xie^{-}+\zeta, t) - \varphi(\zeta, t)\varphi(\xi-\zeta, t) \right] \rd\zeta \rd\sigma,
\end{equation}
where, for the sake of convenience, we introduce extra parameters $ a_{+} = \frac{1+e}{2} $ and $ a_{-} = \frac{1-e}{2} $ such that, 
\begin{equation}\label{xie+}
\xie^{+} = \left( \frac{1}{2} + \frac{\am}{2} \right)\xi + \frac{\ap}{2}|\xi|\sigma,
\end{equation}
\begin{equation}\label{xie-}
\xie^{-} = \left( \frac{1}{2} - \frac{\am}{2} \right)\xi - \frac{\ap}{2}|\xi|\sigma,
\end{equation}
Note that these two vectors $ \xie^{+} $ and $ \xie^{-} $ in the inelastic case satisfy the following relationship,
\begin{equation}
\xie^{+} + \xie^{-} = \xi, \quad |\xie^{+}|^{2} + |\xie^{-}|^{2} = \frac{1+\ap^{2}+ \am^{2}}{2}|\xi|^{2} + \ap\am|\xi|^{2} \frac{\xi\cdot\sigma}{|\xi|}.
\end{equation}
\begin{remark}
	To check this, one can compare with the elastic case by selecting $ e=1 $, which implies that $ \ap=1$ and $\am=0 $, then we find that 
	\begin{equation}
	\xi^{+} = \frac{\xi + |\xi|\sigma}{2}, \quad \xi^{-} = \frac{\xi - |\xi|\sigma}{2}, \quad |\xi^{+}|^{2} + |\xi^{-}|^{2} =|\xi|^{2}.
	\end{equation}
	which coincides with the well-known relations of elastic case.
\end{remark}

In the rest of this section, by first studying the well-posed properties of the solution to Cauchy problem \eqref{IBE} with respect to the following initial condition
\begin{equation}\label{initial}
\p(\xi, 0) = \p_{0}(\xi) \in \mathcal{K}^{\alpha}
\end{equation}
in the space of characteristic functions, we are able to finally obtain the existence and uniqueness of the original Boltzmann equation \eqref{IBEcut}-\eqref{F0cut} under cutoff assumption thanks to the inverse Fourier transformation and the embedded relation \eqref{embed}.

\subsection{Preliminary Lemma}
As in \cite{morimoto2016measure}, the equation \eqref{IBE} can be reformulated as following
\begin{equation}\label{phiequ}
\partial_{t} \varphi(t, \xi) + A\varphi(t, \xi) = \mathcal{G}^{e}_{1} [\varphi] \left(t, \xi\right) + \mathcal{G}^{e}_{2}[\varphi] \left(t, \xi\right)
\end{equation}
where $ A = \sup\limits_{q\in\mathbb{R}^{3}}\left|\Phi_{c}\left(\left|q\right| \right) \right| $ and 
\begin{align}
\mathcal{G}^{e}_{1} [\varphi] \left(t, \xi\right)  = &  \int_{\Sd^{2}} b_{c} \left(\frac{\xi\cdot\sigma}{|\xi|}\right) \int_{\bR^{3}} \hat{\Phi}_{c}(\zeta) \left[ \varphi(t, \xie^{+}-\zeta)\varphi(t, \xie^{-}+\zeta)\right]\,\rd\zeta \,\rd\sigma;\label{G1}\\
\mathcal{G}^{e}_{2} [\varphi] \left(t, \xi\right)  = & A\varphi(t, \xi)-\int_{\Sd^{2}} b_{c} \left(\frac{\xi\cdot\sigma}{|\xi|}\right) \int_{\bR^{3}} \hat{\Phi}_{c}(\zeta) \left[\varphi(t, \zeta)\varphi(t, \xi-\zeta) \right]\,\rd\zeta \,\rd\sigma. \label{G2}
\end{align}
Hence, we can further obtain the following integral form of the solution $ \varphi(t, \xi) $,
\begin{equation}
\varphi(t, \xi) = \e^{-At}\varphi_{0}(\xi) + \int_{0}^{t} \e^{-A(t-\tau)} \left( \mathcal{G}^{e}_{1}[\varphi] \left(\tau, \xi\right) + \mathcal{G}^{e}_{2}[\varphi] \left(\tau, \xi\right)  \right) \,\rd\tau.
\end{equation}

Before getting into proof of Theorem \ref{local}, we first introduce the following technical Lemmas. 

Based on the Lemma \ref{L1} above as well as some elementary inequalities, we give the following estimates, which will play a key role in controlling the inelastic Fourier variables $ \xie^{+} $ and $ \xie^{-} $.
\begin{lemma}\label{bound} 
	Let $ \xie^{+} $ and $ \xie^{-} $ be the vectors defined as \eqref{xie+} and \eqref{xie-} respectively, then for $ \alpha\in[0,2] $, we have
	\begin{equation}\label{bound+}
	\left[ \frac{\ap(1+\am)}{2} \right]^{\frac{\alpha}{2}} \left( 1+\frac{\xi\cdot\sigma}{|\xi|}\right)^{\frac{\alpha}{2}} \left|\xi\right|^{\alpha} \leq \left|\xie^{+}\right|^{\alpha} \leq \left[ \left(\frac{1+\am}{2}\right)^{2}+ \left(\frac{\ap}{2}\right)^{2}\right]^{\frac{\alpha}{2}} \left( 1+\frac{\xi\cdot\sigma}{|\xi|}\right)^{\frac{\alpha}{2}} \left|\xi\right|^{\alpha},
	\end{equation}
	and
	\begin{equation}\label{bound-}
	\left|\xie^{-}\right|^{\alpha} = \left(\frac{\ap^{2}}{2}\right)^{\frac{\alpha}{2}} \left( 1-\frac{\xi\cdot\sigma}{|\xi|}\right)^{\frac{\alpha}{2}} \left|\xi\right|^{\alpha}.
	\end{equation}
\end{lemma}
\begin{proof}
	See our preceding work \cite[Lemma 3.4]{Qi20} for complete proof.
\end{proof}


Furthermore, to obtain the existence and uniqueness of measure-valued solution by Banach Contraction Theorem, such as the elastic case as \cite[Lemma 2.6]{morimoto2016measure}, the following ``inelastic version" Lemma \ref{GG} is given by more delicate analysis.
\begin{lemma}\label{GG}
	 For any restitution coefficient $ e\in (0,1] $ and characteristic function $ \p \in \mathcal{K} $, both of $ \mathcal{G}^{e}_{1}[\varphi] $ and $ \mathcal{G}^{e}_{2}[\varphi] $, defined by \eqref{G1} and \eqref{G2} respectively, are continuous and positive-definite for any $ \varphi\in\mathcal{K} $. Furthermore, if $ \gamma < 0 $ and $ 0< \alpha \leq 2 $, 
	 then for any characteristic functions $ \varphi, \tilde{\p} \in \mathcal{K}^{\alpha}$, there exists a constant $ C_{e}>0 $ such that
	 \begin{equation}\label{GG12}
	 \left|\mathcal{G}^{e}_{1}[\varphi] + \mathcal{G}^{e}_{2}[\varphi] - \mathcal{G}^{e}_{1}[\tilde{\p}] - \mathcal{G}^{e}_{2}[\tilde{\p}] \right| \leq \left( A+C_{e}\right) \left\| \varphi - \tilde{\p} \right\|_{\alpha} \left| \xi\right|^{\alpha},
	 \end{equation}
	 holds for all $ \xi\in\mathbb{R}^{3} $.
\end{lemma}
\begin{proof}
	\textit{Step 1: } According to the Theorem \ref{character1} and Lemma \ref{character2}, it suffices to show that $ \mathcal{G}^{e}_{1}[\varphi] $ is continuous and positive-definite by proving $ \mathcal{G}^{e}_{1}[\varphi] $ is a characteristic function for any $ \varphi\in\mathcal{K} $. Let
	\begin{equation}
	\mathcal{G}^{e,n}_{1}[\varphi]\left(t, \xi\right) =  \int_{\Sd^{2}} b_{c} \left(\frac{\xi\cdot\sigma}{|\xi|}\right) \int_{\bR^{3}} \hat{\Phi}_{c}(\zeta) \left[ \varphi(t, \xie^{+}-\zeta) \e^{-\frac{|\xi^{+}_{e} - \zeta|^{2}}{2n}}   \varphi(t, \xie^{-}+\zeta) \e^{-\frac{|\xi^{-}_{e}+\zeta|^{2}}{2n}} \right] \,\rd\zeta \,\rd\sigma.
	\end{equation}
	Since both of the angular collision kernel $ b_{c} $ and kinetic collision kernel $ \hat{\Phi}_{c} $ are integrable, together with the Lemma \ref{bound} for $ \xi^{+}_{e} $ and $ \xi^{-}_{e} $, we're able to apply the Lebesgue dominated convergence theorem to obtain that
	\begin{equation}
	\mathcal{G}^{e,n}_{1}[\varphi]\left(\cdot, \xi\right) \rightarrow \mathcal{G}^{e}_{1}[\varphi]\left(\cdot, \xi\right) \ \text{pointwisely as}\ n\rightarrow\infty
	\end{equation}
	Then, following the proof of \cite[Lemma 2.1]{PTfourier1996}, we only need to show that $ \mathcal{G}^{e,n}_{1}[\varphi] $ is a characteristic function with respect to $ \xi $ for each $ n $. Denote $ F(v) = \mathcal{F}^{-1}(\p) $, then $ \mathcal{G}^{e,n}_{1}[\varphi](\cdot,\xi) $ is the Fourier transform of the following positive density
	\begin{equation}
	G_{1}^{e,n}[F](\cdot,v) = \int_{\bR^{3}} \int_{\Sd^{2}} b_{c}\left(\frac{v-v_{*}}{|v-v_{*}|} \cdot \sigma\right) \Phi_{c}\left(|v-v_{*}|\right) f_{n}(\cdot, v'_{*}) f_{n}(\cdot, v') \,\rd \sigma \,\rd v_{*},
	\end{equation}
	where $ v' $ and $ v'_{*} $ are given by \eqref{ve} and
	\begin{equation}\label{ff}
	f_{n}(v) = \int_{\bR^{3}} \omega_{n}(v-u) \,\rd F(u), \quad \omega_{n}(v) = \left(\frac{n}{2\pi}\right)^{\frac{3}{2}} \e^{-\frac{n}{2}|v|^{2}}.
	\end{equation}
	On the other hand, note that $ \mathcal{G}^{e}_{2}[\varphi] $ actually has the same definition as the elastic counterpart $ \mathcal{G}_{2}[\varphi] $, and one can refer to the \cite[Lemma 2.6]{morimoto2016measure} for the proof of $ \mathcal{G}^{e}_{2}[\varphi] $. This completes the proof of the continuity and positive-definite property of $ \mathcal{G}^{e}_{1}[\varphi] $ and $ \mathcal{G}^{e}_{2}[\varphi] $ for any $ \varphi\in\mathcal{K} $.
	
	\textit{Step 2:} To prove the estimate \eqref{GG12}, we start from substituting the $ \varphi $ and $ \tilde{\p} $ into \eqref{G1}-\eqref{G2} and taking the subtraction
	\begin{equation}\label{soft1}
	\begin{split}
	&\left|\mathcal{G}^{e}_{1}[\varphi] + \mathcal{G}^{e}_{2}[\varphi] - \mathcal{G}^{e}_{1}[\tilde{\p}] - \mathcal{G}^{e}_{2}[\tilde{\p}] \right| \\
	\leq& \int_{\Sd^{2}} b_{c} \left(\frac{\xi\cdot\sigma}{|\xi|}\right) \int_{\bR^{3}} \left| \hat{\Phi}_{c}(\zeta - \xie^{-}) - \hat{\Phi}_{c}(\zeta)\right|  \left| \varphi(\zeta)\varphi(\xi-\zeta) - \tilde{\p}(\zeta)\tilde{\p}(\xi-\zeta)\right| \,\rd\zeta \,\rd\sigma\\
	&+ A \left\| \varphi - \tilde{\p} \right\|_{\alpha} \left| \xi\right|^{\alpha},
	\end{split}
	\end{equation}
	where we utilize the change of variable $ \xie^{-} + \zeta \rightarrow \zeta $ and the relation $ \xie^{+} + \xie^{-} = \xi $ for \eqref{G1}.\\
       By noticing that $ \hat{\Phi}_{c}(\zeta) $ is a real-valued function and satisfies $ \hat{\Phi}_{c}(\zeta) = \hat{\Phi}_{c}(-\zeta) $, we apply the transformation $ \zeta \mapsto \xi - \zeta $ and find that
	\begin{equation}\label{need1}
	\begin{split}
	&\int_{\Sd^{2}} b_{c} \left(\frac{\xi\cdot\sigma}{|\xi|}\right) \int_{\bR^{3}} \left[ \hat{\Phi}_{c}(\zeta - \xie^{-}) - \hat{\Phi}_{c}(\zeta)\right]  \varphi(\zeta) \varphi(\xi- \zeta)  \,\rd\zeta \,\rd\sigma\\
	= & \int_{\Sd^{2}} b_{c} \left(\frac{\xi\cdot\sigma}{|\xi|}\right) \int_{\bR^{3}} \left[ \hat{\Phi}_{c}(\xi - \zeta - \xie^{-}) - \hat{\Phi}_{c}(\xi - \zeta)\right]  \varphi(\xi - \zeta) \varphi(\zeta)  \,\rd(\xi - \zeta) \,\rd\sigma\\
	= & \int_{\Sd^{2}} b_{c} \left(\frac{\xi\cdot\sigma}{|\xi|}\right) \int_{\bR^{3}} \left[ \hat{\Phi}_{c}(\xi - \zeta) - \hat{\Phi}_{c}(\xi - \zeta - \xie^{-}) \right]  \varphi(\xi - \zeta) \varphi(\zeta)  \,\rd\zeta \,\rd\sigma\\
	= & \frac{1}{2} \int_{\Sd^{2}} b_{c} \left(\frac{\xi\cdot\sigma}{|\xi|}\right) \int_{\bR^{3}} \left[ \hat{\Phi}_{c}(\zeta - \xie^{-}) + \hat{\Phi}_{c}(\xi - \zeta) - \hat{\Phi}_{c}(\xi - \zeta - \xie^{-}) - \hat{\Phi}_{c}(\zeta) \right]  \varphi(\xi - \zeta) \varphi(\zeta)  \,\rd\zeta \,\rd\sigma
	\end{split}
	\end{equation}
	which also works if replacing $ \p $ by $ \tilde{\p} $.\\
	Additionally, by the direct calculation, we have
	\begin{equation}\label{need2}
	\begin{split}
	&\left| \varphi(\zeta)\varphi(\xi-\zeta) - \tilde{\p}(\zeta)\tilde{\p}(\xi-\zeta)\right|\\
	\leq & \left| \varphi(\zeta)\varphi(\xi-\zeta) - \varphi(\zeta)\tilde{\p}(\xi-\zeta) + \varphi(\zeta)\tilde{\p}(\xi-\zeta) - \tilde{\p}(\zeta)\tilde{\p}(\xi-\zeta) \right|\\
	\leq &\left| \varphi(\zeta) \right| \left\|\varphi -\tilde{\p} \right\|_{\alpha} \left|\xi-\zeta \right|^{\alpha} + \left\|\varphi -\tilde{\p} \right\|_{\alpha} \left| \zeta \right|^{\alpha} \left|\tilde{\p}(\xi-\zeta)\right|\\
	\leq & \frac{\left|\xi-\zeta \right|^{\alpha} + \left| \zeta \right|^{\alpha}}{\left| \xi \right|^{\alpha}} \left\|\varphi -\tilde{\p} \right\|_{\alpha} \left| \xi \right|^{\alpha},
	\end{split}
	\end{equation}
	where the property of the characteristic functions that $ \left| \varphi(\zeta) \right| < 1 $ and $ \left|\tilde{\p}(\xi-\zeta)\right| < 1 $ is utilized in the last inequality.\\
       Hence, combing the estimates \eqref{need1} and \eqref{need2} , it only remains to show that
	\begin{equation}\label{softpur}
	\int_{\bR^{3}} \left| \hat{\Phi}_{c}(\zeta - \xie^{-}) + \hat{\Phi}_{c}(\xi - \zeta) - \hat{\Phi}_{c}(\xi - \zeta - \xie^{-}) - \hat{\Phi}_{c}(\zeta) \right|  \frac{\left|\xi-\zeta \right|^{\alpha} + \left| \zeta \right|^{\alpha}}{\left| \xi \right|^{\alpha}}  \,\rd\zeta \leq C_{e}.
	\end{equation}
	In fact, the $ |\xi| > 1 $ case can be directly obtained, since no singularity exists in the estimate above. Then, for $ |\xi| \leq 1 $, we find that
	\begin{equation}\label{soft11}
	\begin{split}
	&\left| \hat{\Phi}_{c}(\zeta - \xie^{-})  - \hat{\Phi}_{c}(\zeta) - \nabla\hat{\Phi}_{c}(\zeta)\cdot \xie^{-} \right|\\
	= &\left| \int_{0}^{1} (1-\tau)(\xie^{-})^{T} \nabla^{2}\hat{\Phi}_{c}(-\zeta + \tau\xie^{-})(\xie^{-}) \,\rd\tau  \right|\\
	\lesssim & \frac{\ap^{2}}{2} |\xi|^{2} \int_{0}^{1} \frac{1}{\left\langle -\zeta + \tau\xie^{-} \right\rangle^{2N} } \,\rd\tau,
	\end{split}
	\end{equation}
	and
	\begin{equation}\label{soft12}
	\begin{split}
	&\left| \hat{\Phi}_{c}(\xi - \zeta) - \hat{\Phi}_{c}(\xi - \zeta - \xie^{-}) -\nabla\hat{\Phi}_{c}(-\zeta+\xi)\cdot (-\xie^{-})  \right|\\
	= &\left| \int_{0}^{1} (1-\tau)(-\xie^{-})^{T} \nabla^{2}\hat{\Phi}_{c}(-\zeta +\xi - \tau\xie^{-})(-\xie^{-}) \,\rd\tau  \right|\\
	\lesssim & \frac{\ap^{2}}{2} |\xi|^{2} \int_{0}^{1} \frac{1}{\left\langle -\zeta + \xi - \tau\xie^{-} \right\rangle^{2N} } \,\rd\tau.
	\end{split}
	\end{equation}
	where we use the fact $ |\xie^{-}|^{2} \leq \frac{\ap^{2}}{2} |\xi|^{2} $ and Lemma \ref{PhiDecay} in the two inequalities above.\\
	As a consequence, combining \eqref{soft11} and \eqref{soft12}, we further obtain
	\begin{equation}
	\begin{split}
	&\left| \hat{\Phi}_{c}(\zeta - \xie^{-}) + \hat{\Phi}_{c}(\xi - \zeta) - \hat{\Phi}_{c}(\xi - \zeta - \xie^{-}) - \hat{\Phi}_{c}(\zeta) \right|\\
	 \lesssim& \left| \nabla\hat{\Phi}_{c}(\zeta)\cdot \xie^{-} +  \nabla\hat{\Phi}_{c}(-\zeta+\xi)\cdot (-\xie^{-})  \right| + \frac{\ap^{2}}{2} |\xi|^{2} \int_{0}^{1} \frac{1}{\left\langle -\zeta + \tau\xie^{-} \right\rangle^{2N} } + \frac{1}{\left\langle -\zeta + \xi - \tau\xie^{-} \right\rangle^{2N} } \,\rd\tau
	\end{split}
	\end{equation}
	Moreover, considering the following fact that, since $ |\xie^{-}|^{2} \leq \frac{\ap^{2}}{2} |\xi|^{2}  $,
	\begin{equation}
	\begin{split}
	\left|  \nabla\hat{\Phi}_{c}(\zeta)\cdot \xie^{-} +  \nabla\hat{\Phi}_{c}(-\zeta+\xi)\cdot (-\xie^{-}) \right| =& \left| \int_{0}^{1} (\xie^{-})^{T} \nabla^{2}\hat{\Phi}_{c}(-\zeta + \tau\xie^{-})(\xie^{-}) \,\rd\tau  \right|\\
	\leq &  \frac{\ap^{2}}{2} |\xi|^{2} \int_{0}^{1} \frac{1}{\left\langle -\zeta + \tau\xie^{-} \right\rangle^{2N} } \,\rd\tau,
	\end{split}
	\end{equation}
	we finally conclude that, if $  |\xi| \leq 1  $,
	\begin{equation}
	\begin{split}
	&\left| \hat{\Phi}_{c}(\zeta - \xie^{-}) + \hat{\Phi}_{c}(\xi - \zeta) - \hat{\Phi}_{c}(\xi - \zeta - \xie^{-}) - \hat{\Phi}_{c}(\zeta) \right|\\
	\lesssim& \frac{\ap^{2}}{2} |\xi|^{2} \int_{0}^{1} \frac{1}{\left\langle -\zeta + \tau\xie^{-} \right\rangle^{2N} } + \frac{1}{\left\langle -\zeta + \xi - \tau\xie^{-} \right\rangle^{2N} } + \frac{1}{\left\langle -\zeta + \tau\xie^{-} \right\rangle^{2N} } \,\rd\tau\\
	\leq & C_{e}
	\end{split}
	\end{equation}
	This completes the proof of \eqref{softpur}, which further implies our desired estimate \eqref{GG12} with the help of \eqref{soft1}.
\end{proof}

\subsection{Proof of the Theorem \ref{local}}

In this subsection, the Banach Contraction Theorem is applied to prove the existence and uniqueness of the solution to cutoff equation \eqref{phiequ}. To achieve this, we define the following non-linear operator: for fixed initial datum $ \p_{0}(\xi) \in \K^{\alpha}$,
\begin{equation}\label{P}
\mathcal{P}[\p]\left(t,\xi\right) := \p_{0}(\xi) \e^{-At} + \int_{0}^{t} \e^{-A(t-\tau)} \left( \mathcal{G}^{e}_{1} [\p] \left(\tau, \xi\right) + \mathcal{G}^{e}_{2}[\p] \left(\tau, \xi\right) \right) \,\rd\tau.
\end{equation}
Then we're ready to prove the local existence and uniqueness by showing that operator $\mathcal{P} : \chi^{\alpha}_{T} \mapsto C( [0,T], \K^{\alpha} )$ has a unique fixed point in the space  $\chi^{\alpha}_{T} \subset C( [0,T], \K^{\alpha} )$ defined as
\begin{equation}
\chi^{\alpha}_{T} := \left\lbrace  \p\in C\left( [0,T],\K^{\alpha} \right) : \sup\limits_{t\in [0,T]}\left\|\p(t, \cdot)\right\|_{\alpha} < \infty \right\rbrace,
\end{equation}
which is a complete metric space with respect to the induced norm
\begin{equation}
\left\| \cdot \right\|_{\chi^{\alpha}_{T}}: =\sup\limits_{t\in [0,T]} \left\| \cdot \right\|_{\alpha}.
\end{equation}

\begin{proof}
For any $ \p(t,\xi) \in \mathcal{K} $, the continuity and positivity of $ \mathcal{P}[\p](t,\xi) $ can be directly obtained. Indeed, from the Lemma \ref{GG}, we can find that $ \mathcal{G}^{e}_{1}[\p] + \mathcal{G}^{e}_{2}[\p]$ is continuous and positive-definite for every $ \tau\in\left[0,t\right] $; it then follows from Lemma \ref{character2} that $ \mathcal{P}[\p](t,\xi) \in \K^{\alpha}$ .\\
(i) To show that the non-linear operator $ \mathcal{P} $ maps $ \chi^{\alpha}_{T} $ into itself. By observing that $ \mathcal{G}^{e}_{1}[\p] (t,0) + \mathcal{G}^{e}_{2}[\p] (t,0) = A $, we have
\begin{equation}
\mathcal{P}[\p](t,0) =1 
\end{equation}
and then, for every $  \p\in \chi^{\alpha}_{T} $,
\begin{equation}
\begin{split}
\mathcal{P}[\p](t,\xi) - 1 = \e^{-At}(\p_{0} - 1) + \int_{0}^{t}& \e^{-A(t-\tau)} ( \mathcal{G}^{e}_{1}[\p] \left(\tau, \xi\right) + \mathcal{G}^{e}_{2}[\p] \left(\tau, \xi\right)\\
& - \mathcal{G}^{e}_{1}[1](\tau,\xi) - \mathcal{G}^{e}_{2}[1](\tau,\xi) ) \,\rd\tau
\end{split}
\end{equation}
Furthermore, by Lemma \ref{GG}, we find that, for any $ \tau \in (0,t) $,
\begin{equation}
\left| \mathcal{P}[\p](\xi,t) - 1 \right| \leq \left\| \p_{0} - 1 \right\|_{\alpha} \left| \xi \right|^{\alpha} + \left( A + C_{e} \right) \int_{0}^{t} \left\| \p(\xi,\tau) - 1 \right\|_{\alpha} \ \rd \tau \left| \xi \right|^{\alpha}
\end{equation} 
After dividing the inequality above by $ |\xi|^{\alpha} $ and computing the supremum with respect to the variable $ \xi\in\mathbb{R}^{3} $ and $ t\in\left[0,T\right] $, we then obtain that
\begin{equation}
\sup_{t \in [0,T]} \| \mathcal{P}[\p](t,\cdot) - 1 \|_{\alpha} \leq \| \p_{0} -1\|_{\alpha} + \left( A + C_{e} \right) T \| \p - 1 \|_{\chi^{\alpha}_{T}} < \infty
\end{equation}
which implies that the operator $ \mathcal{P} $ maps $ \chi^{\alpha}_{T} $ into itself.\\
(ii) To prove that $ \mathcal{P}[\p] \in \K^{\alpha}$ is a contraction in $\chi^{\alpha}_{T}$, we introduce another $ \mathcal{P}[\tilde{\p}] \in \K^{\alpha} $, and make the subtraction between them. Then for the same initial datum $ \p_{0} $, we have, 
\begin{equation}
\begin{split}
\left| \mathcal{P}[\p](t,\xi) - \mathcal{P}[\tilde{\p}](t,\xi) \right| \leq& \int_{0}^{t} \e^{-A(t-\tau)} (\mathcal{G}^{e}_{1}[\p] (\tau, \xi) + \mathcal{G}^{e}_{2}[\p] (\tau, \xi) \\& - \mathcal{G}^{e}_{1}[\tilde{\p}] (\tau, \xi) -\mathcal{G}^{e}_{2}[\tilde{\p}] (\tau, \xi) ) \,\rd\tau\\
\leq& ( A + C_{e} ) T \left\|\p - \tilde{\p}  \right\|_{\chi^{\alpha}_{T}} \left| \xi \right|^{\alpha}
\end{split}
\end{equation}
where we utilize the Lemma \ref{GG} in the last inequality. Consequently, after dividing the inequality above by $ |\xi|^{\alpha} $ with respect to the variable $ \xi\in\mathbb{R}^{3} $, we can obtain
\begin{equation}
\left\| \mathcal{P}[\p] - \mathcal{P}[\tilde{\p}] \right\|_{\chi^{\alpha}_{T}} \leq ( A + C_{e} ) T \left\|\p - \tilde{\p}  \right\|_{\chi^{\alpha}_{T}}.
\end{equation}
Combining (i) and (ii) above, the Banach Contraction Theorem admits the unique solution of \eqref{IBEcut} in the space $ \chi^{\alpha}_{T} $ provided that $ T < \frac{1}{A+C_{e}}. $

Note that we construct the unique solution on the time interval $ \left[ 0 ,T\right] $, where $ T $ is independent of the initial datum, therefore, by the continuation argument, we can extend the unique solution to $ \left[ T,2T\right] $ by choosing $ \p(T,\xi) $ as the initial datum. Consequently, by repeating the same procedure, it suffices to conclude that there exists an unique solution $ \p(t,\xi) \in C( [0,\infty), \K^{\alpha} ) $ to the Cauchy problem \eqref{IBE} associated with any initial datum \eqref{initial}. 

As a result, by noticing the inclusive relation $ \mathcal{F}^{-1}(\mathcal{K}^{\alpha_{0}}) \subset P_{\alpha} $, we can finally prove the existence and uniqueness of the solution to the original Cauchy problem \eqref{IBEcut}-\eqref{F0cut} under cutoff assumption.
\end{proof}

\subsection{Conservative and Dissipative Property}
In this subsection, the conservative and dissipative properties of the measure-valued solution obtained by the Theorem \ref{local} will be derived in the following Corollary \ref{DissCutoff}.
\begin{corollary}\label{DissCutoff}
	For the $ F_{t} \in C\left(  \left[0,\infty\right), P_{\alpha} \right) $, which is the unique solution to Cauchy problem \eqref{IBEcut}-\eqref{F0cut} constructed in Theorem \ref{local}, if the initial datum $ F_{0} \in P_{\alpha_{0}}(\mathbb{R}^{3}) $ with $ \alpha_{0}\in \left(0,2\right] $, it satisfies
	\begin{equation}\label{energydiss}
	\forall t \geq 0, \quad \int_{\bR^{3}} |v|^{2} \,\rd F_{t}(v) \leq \int_{\bR^{3}} |v|^{2} \,\rd F_{0}.
	\end{equation}
       furthermore, if the initial datum $ F_{0} \in P_{\alpha_{0}}(\mathbb{R}^{3}) $ with $ \alpha_{0} \geq 1 $, it satisfies
       \begin{equation}\label{momentumcon}
       \forall t \geq 0, \quad \int_{\bR^{3}} v_{j} \,\rd F_{t}(v) = \int_{\bR^{3}} v_{j}\,\rd F_{0}, \quad j=1,2,3.
       \end{equation}
\end{corollary}
\begin{proof}
	Similar to the proof of elastic case in \cite{morimoto2016measure}, we introduce the weighted function $ W_{\delta}(v) = \left\langle v \right\rangle^{\alpha_{0}} \left\langle \delta v \right\rangle^{-\alpha_{0}} $ for $ \delta > 0 $, and since $ \alpha_{0}\in \left(0,2\right] $ and $ |v'|^{2} \leq |v|^{2} + |v_{*}|^{2} $ still holds in the inelastic collision process, if follows that
	\begin{equation}\label{W1}
	\begin{split}
	 W_{\delta}(v') = \left( \frac{1+|v'|^{2}}{1 + \delta^{2}|v'|^{2}} \right)^{\frac{\alpha_{0}}{2}} \leq& \left( \frac{ 1 + |v|^{2} + |v_{*}|^{2}}{ 1 + \delta^{2}\left(|v|^{2} +|v_{*}|^{2}\right)} \right)^{\frac{\alpha_{0}}{2}}\\
	 \leq& \left( \frac{1+ |v|^{2} }{1 + \delta^{2}|v|^{2} } +\frac{1+|v_{*}|^{2}}{1 + \delta^{2}|v_{*}|^{2}} \right)^{\frac{\alpha_{0}}{2}}\\
	 \lesssim & W_{\delta}(v) + W_{\delta}(v_{*})
	\end{split}
	\end{equation}
	Then, by the derivation of weak formulation, we can obtain that
	\begin{equation}
	\begin{split}
	&\frac{\rd}{\rd t} \int_{\bR^{3}} W_{\delta}(v) \,\rd F_{t}(v)\\
	= & \int_{\bR^{3}} \int_{\bR^{3}} \int_{\Sd^{2}} b_{c} \left(\frac{v-v_{*}}{|v-v_{*}|} \cdot \sigma\right) \Phi_{c}(|v-v_{*}|) W_{\delta}(v) \,\rd \sigma  \left[ \,\rd F_{t}(v') \,\rd F_{t}(v'_{*}) -\,\rd \sigma \,\rd F_{t}(v) \,\rd F_{t}(v_{*}) \right]\\
	\leq & \int_{\bR^{3}} \int_{\bR^{3}} \int_{\Sd^{2}} b_{c} \left(\frac{v-v_{*}}{|v-v_{*}|} \cdot \sigma\right) \Phi_{c}(|v-v_{*}|) W_{\delta}(v') \,\rd \sigma \,\rd F_{t}(v) \,\rd F_{t}(v_{*})\\
	\lesssim & \int_{\bR^{3}} \int_{\bR^{3}} \int_{\Sd^{2}} b_{c} \left(\frac{v-v_{*}}{|v-v_{*}|} \cdot \sigma\right) \Phi_{c}(|v-v_{*}|) \left[ W_{\delta}(v) + W_{\delta}(v_{*}) \right] \,\rd \sigma \,\rd F_{t}(v) \,\rd F_{t}(v_{*})\\
	\lesssim & \int_{\bR^{3}} W_{\delta}(v)\,\rd F_{t}(v)
	\end{split}
	\end{equation}
	where the estimate \eqref{W1} is utilized in the second inequality above.\\
	By applying the Gronwall's inequality, it further implies that
	\begin{equation}
	\int_{\bR^{3}} W_{\delta}(v)\,\rd F_{t}(v) < C'_{e}.
	\end{equation}
	Then, by letting $ \delta \rightarrow 0 $, we obtain $ \int_{\bR^{3}} \left\langle v \right\rangle^{\alpha_{0}} \,\rd F_{t}(v) < \infty $. Finally, by observing the property \eqref{QEcon}, the \eqref{momentumcon} holds since $ v_{j}, j=1,2,3 $ are collisional invariants in the inelastic process, while \eqref{energydiss} holds by using the energy dissipation estimate \eqref{QEdiss} of the inelastic Boltzmann operator.
\end{proof}

\section{Existence without Cutoff Assumption}\label{sec:noncutoff}
In this section, we are prepared to consider the Cauchy problem of inelastic Boltzmann equation \eqref{IB}-\eqref{F0} without the cutoff assumption, which implies that
\begin{equation}\label{noncutoff}
\int_{\Sd^{2}} b\left(\frac{v-v_{*}}{|v-v_{*}|} \cdot \sigma\right) \,\rd \sigma = \infty,
\end{equation}
more precisely, $ b $ satisfies the singularity condition \eqref{noncutoffb}, and kinetic collision kernel $ \Phi(|v-v_{*}|) = |v-v_{*}|^{\gamma} $ with $ -2 \leq \gamma <0 $.

In fact, our strategy is to construct the solutions to \eqref{IB}-\eqref{F0} with non-cutoff collision kernel based on compactness argument, hence, we first consider the increasing sequence of bounded collision kernels, that is, for each $  n \geq 2 $
\begin{equation}
b_{n}\left(\frac{v-v_{*}}{|v-v_{*}|} \cdot \sigma\right) := \min \left\lbrace b\left(\frac{v-v_{*}}{|v-v_{*}|} \cdot \sigma\right), n \right\rbrace \leq b\left(\frac{v-v_{*}}{|v-v_{*}|} \cdot \sigma\right), \quad n\in\mathbb{N},
\end{equation}
where $ \Phi_{n}(|v-v_{*}|) = \Phi(|v-v_{*}|) \phi_{n}(|v-v_{*}|) $ for the moderately soft potential case $ -2 \leq \gamma <0 $ .

\subsection{Finite Energy Initial Datum}
In this subsection, we will discuss the finite energy initial datum case, \textit{i.e.}, $ F_{0} \in P_{2}(\mathbb{R}^{3}) $. By Theorem \ref{local}, for any initial datum $ F_{0} \in P_{2}(\mathbb{R}^{3}) $, there exists a unique measure-valued solution $ F_{t}^{n} $ to the following cutoff equation,
\begin{equation}\label{eqpcuto}
\partial_{t} f(t, v) = \int_{\bR^{3}} \int_{\Sd^{2}} b_{n}\left(\frac{v-v_{*}}{|v-v_{*}|} \cdot \sigma\right) \Phi_{n}(|v-v_{*}|)  \left[ J f(\vs')f(v')- f(\vs)f(v) \right] \,\rd \sigma \,\rd \vs,
\end{equation}
hence, note that the corresponded characteristic function $ \p^{n}(t,\xi) = \mathcal{F}(F_{t}^{n}) \in C\left([0,\infty), \mathcal{K}^{\alpha} \right) $ satisfies the following equation,
\begin{equation}\label{eqpcut}
\partial_{t} \varphi(t,\xi) = \int_{\Sd^{2}} b_{n}\left(\frac{\xi\cdot\sigma}{|\xi|}\right) \int_{\bR^{3}} \hat{\Phi}_{n}(\zeta) \left[ \varphi(t,\xie^{+}-\zeta)\varphi(t,\xie^{-}+\zeta) - \varphi(t,\zeta)\varphi(t,\xi-\zeta) \right]\,\rd\zeta \,\rd\sigma
\end{equation}
with the initial datum $ \p_{0}(\xi) = \mathcal{F}(F_{0}) $.

\subsubsection{Preliminary}

\begin{proposition}\label{finitePro}
	For $  e \in (0,1] $ and $ -2 \leq \gamma <0 $. let the angular collision kernel $ b $ satisfy the non-cutoff assumption \eqref{noncutoffb}, then the sequence of solutions $ \left\lbrace \p^{n}(t,\xi) \right\rbrace_{n=1}^{\infty}  $ to \eqref{eqpcut} are bounded in $ C\left( [0,\infty) \times \mathbb{R}^{3} \right) $ and equi-continuous on any compact subsets of $ [0,\infty) \times \mathbb{R}^{3} $.
\end{proposition}
\begin{proof}
\textit{Step 1: Uniform Bound.} According to Theorem \ref{local}, the sequence of solution $ \p^{n}\left(\cdot,\xi\right) \in \K^{\alpha} $ under cutoff assumption are all characteristic function for every $ t\geq 0 $, hence, we have 
\begin{equation}
\left| \p^{n}\left(  t,\xi \right) \right| \leq  1
\end{equation}
for all $ \xi\in\mathbb{R}^{3} $ and $ t\geq 0 $, which implies the uniform bound of $ \p^{n}\left(  t,\xi \right) $.\\
\textit{Step 2: Equi-continuity in the Fourier variable $ \xi $.} To prove the equi-continuity with respect to variable $ \xi $ in a compact set of $ \mathbb{R}^{3} $, it suffices to apply Lemma \ref{L1}, combined with \cite[Lemma 3.3]{Qi20} to obtain the following estimate: for all $ t\geq 0 $,
\begin{equation}
\begin{split}
\left| \p^{n} (t,\xi) - \p^{n}(t,\eta) \right| \leq & \sqrt{2\left[ 1- \text{Re} \p^{n}\left( t,\xi-\eta\right)\right]}\\
\leq & \sqrt{2}\left| \xi-\eta \right|^{\frac{\alpha}{2}} \left\| \p^{n}(t) - 1 \right\|_{\alpha}^{\frac{1}{2}}\\
\leq & \sqrt{2} C_{\alpha}\left| \xi-\eta \right|^{\frac{\alpha}{2}} \sup_{t \in [0,\infty)} \int_{\bR^{3}} |v|^{\alpha} \,\rd F_{t}^{n}\\
\leq & \sqrt{2} C_{\alpha}\left| \xi-\eta \right|^{\frac{\alpha}{2}} \sup_{t \in [0,\infty)} \int_{\bR^{3}} \left( 1 + |v|^{2} \right) \,\rd F_{0}.
\end{split}
\end{equation}
where the energy dissipation Corollary \ref{DissCutoff} is applied for $ F_{t}^{n} $ in the last inequality above.\\
\textit{Step 3: Equi-continuity in the time variable $ t $.} 
To achieve this, by recalling the Definition \ref{measure} of measure-valued solution and letting $ \psi \in C_{b}^{2}(\mathbb{R}^{3}) $ be the test function, it suffices to find that
\begin{equation}\label{equit1}
\begin{split}
& \left|  \int_{\bR^{3}} \psi(v) \,\rd F_{t}^{n}(v) - \int_{\bR^{3}} \psi(v) \,\rd F_{s}^{n}(v) \right|\\
\leq & \frac{1}{2} \int_{s}^{t} \left| \int_{\bR^{3}} \int_{\bR^{3}} \int_{\Sd^{2}} b_{n}\left(\psi'_{*} + \psi' - \psi_{*} - \psi \right)\,\rd \sigma \left|v-v_{*}\right|^{\gamma} \psi_{n}(|v-v_{*}|) \,\rd F_{\tau}^{n}(v) \,\rd F_{\tau}^{n}(v_{*}) \right| \,\rd \tau
\end{split}
\end{equation}
where the shorthand notations $ \psi'_{*} = \psi(v'_{*}), \psi' = \psi(v'), \psi_{*} = \psi(v_{*}), \psi = \psi(v) $ are introduced  for simplicity. Further considering the pre-post collision velocities relation \eqref{ve} that
\begin{equation}
\begin{split}
v'-v = & \frac{v_{*}- v}{2} + \frac{1-e}{4}(v-v_{*}) + \frac{1+e}{4} |v-v_{*}|\sigma\\
= & \frac{1+e}{2} \left[ \frac{1}{2} |v-v_{*}|\sigma - \frac{1}{2}(v-v_{*})   \right]\\
= & a_{+} \frac{|v-v_{*}|}{2} \left[ \sigma - \left(\sigma\cdot\frac{v-v_{*}}{|v-v_{*}|} \right)\frac{v-v_{*}}{|v-v_{*}|} \right]  + a_{+} \frac{v-v_{*}}{2} \left[ \left(\sigma\cdot\frac{v-v_{*}}{|v-v_{*}|}\right) - 1  \right]
\end{split}
\end{equation}
and letting $ \hat{q} = \frac{v-v_{*}}{|v-v_{*}|} $, then the Taylor expansion up to the second order gives that
\begin{equation}
\begin{split}
\psi' - \psi = & \nabla \psi(v)(v'-v) + \int_{0}^{1} (1-\tau) \nabla^{2}\psi \left[v+\tau(v'-v)\right] \,\rd \tau (v'-v)^{2}\\
= & a_{+} \frac{|v-v_{*}|}{2} \nabla\psi(v) \left[\sigma - (\sigma\cdot\hat{q})\hat{q}\right] + a_{+} \nabla\psi(v) \frac{v-v_{*}}{2} \left[(\sigma\cdot\hat{q}) - 1\right] + O\left(|v-v_{*}|^{2} \theta^{2}\right),
\end{split}
\end{equation}
and the similar expansion can be obtained for $ \psi'_{*} - \psi_{*} $ as well, which further implies that
\begin{equation}\label{fourpsi}
\psi'_{*} + \psi' - \psi_{*} - \psi + a_{+} \frac{|v-v_{*}|}{2} \left[\nabla\psi(v) - \nabla\psi(v_{*})\right] \left[(\sigma\cdot\hat{q})\hat{q} - \sigma \right] = O\left(|v-v_{*}|^{2} \theta^{2}\right).
\end{equation}
On the other hand, noticing that the symmetry on the sphere integral that
\begin{equation}\label{sym1}
\begin{split}
&\int_{\Sd^{2}} b_{n}(\sigma\cdot\hat{q}) \left[(\sigma\cdot\hat{q})\hat{q} - \sigma \right] \,\rd\sigma\\
 = & \int_{0}^{2\pi} \int_{0}^{\pi} b_{n}(\cos\theta) \left(\hat{h} \sin\theta\cos\phi + \hat{j} \sin\theta\sin\phi \right) \sin\theta \,\rd\theta \,\rd \phi   = 0
\end{split}
\end{equation}
which can be calculated by decomposing the $ \sigma \in \Sd^{2} $ as
\begin{equation}
\sigma = \hat{q} \cos\theta + \hat{h} \sin\theta\cos\phi + \hat{j} \sin\theta\sin\phi 
\end{equation}
by the following orthogonal basis constructed by $ \hat{q} $ and
\begin{equation}
\hat{j} := \frac{v \times v_{*}}{|v \times v_{*}|}, \qquad \hat{h}:= \hat{j} \times \hat{q} = \frac{( (v-v_{*})\cdot v )v_{*} - ( (v-v_{*})\cdot v_{*} )v}{|v-v_{*}| | v\times v_{*}|}.
\end{equation}
Combing \eqref{fourpsi} and \eqref{sym1}, together with the non-cutoff assumption \eqref{noncutoffb}, it follows that
\begin{equation}\label{OO1}
\left| \int_{\Sd^{2}} b_{n}\left(\psi'_{*} + \psi' - \psi_{*} - \psi \right) \,\rd \sigma \right| \leq C_{e,\psi} |v-v_{*}|^{2}
\end{equation}
As a result, after substituting estimation \eqref{OO1} into \eqref{equit1}, we finally obtain that, for any $ 0 \leq s < t \leq T $, there exists a constant $ C_{e,\psi} > 0 $ (independent of $ n $) such that
\begin{equation}
\begin{split}\label{equitsoft}
\left|  \int_{\bR^{3}} \psi \,\rd F_{t}^{n}(v) - \int_{\bR^{3}} \psi \,\rd F_{s}^{n}(v) \right|
\leq & \left|t-s\right| \frac{C_{e, \psi}}{2} \sup_{s\leq\tau\leq t} \int_{\bR^{3}} \int_{\bR^{3}} \left|v-v_{*}\right|^{\gamma+2} \,\rd F_{\tau}^{n}(v) \,\rd F_{\tau}^{n}(v_{*})\\
\leq & \left|t-s\right| \frac{C_{e, \psi}}{2} \int_{\bR^{3}} \left( 1 + |v|^{2} \right) \,\rd F_{0}.
\end{split}
\end{equation}
since $ -2 \leq \gamma <0 $ in the moderately soft potential case. Since we merely consider the compact subset of $ [0,\infty) \times \mathbb{R}^{3} $, then for any fixed bounded $ \xi\in\mathbb{R}^{3} $, the function $ \e^{\im \xi\cdot v} \in C_{b}^{2}(\mathbb{R}^{3}) $ is sufficient to be the test function. And this completes the proof by selecting $ \psi(v) = \e^{\im \xi\cdot v} $ in \eqref{equitsoft}.
\end{proof}

By the Proposition \ref{finitePro} above, we are able to apply the Ascoli-Arzela Theorem and the Cantor diagonal rule to admit the subsequence of solutions $ \{\p^{n_{k}} (t,\xi) \}^{\infty}_{k=1} \subset \{\p^{n} (t,\xi) \}^{\infty}_{n=1} $, which converges uniformly on every compact subset of $ \mathbb{R} \times [0,\infty) $. Then, by letting
\begin{equation}
\p(t,\xi) = \lim\limits_{n\rightarrow \infty} \p^{n}(t,\xi),
\end{equation}
and $ \p(t,\xi) $ is the point-wise limit of characteristic functions, which are the solution to \eqref{eqpcut}, for every $ t > 0 $. Furthermore, denote $ F^{n}_{t} := \mathcal{F}^{-1}\left( \p^{n}(t,\xi) \right) $ and $ F_{t} := \mathcal{F}^{-1}\left( \p(t,\xi) \right) $, then we have
\begin{equation}\label{weakconF}
F^{n}_{t} \rightharpoonup F_{t}, \quad \forall \psi\in C_{b}^{2}(\mathbb{R}^{3}).
\end{equation}
Thus, the rest is to prove $ F_{t} $ is a measure-valued solution to \eqref{IB}. 

Before that, the following lemma \ref{DissiLemma} is introduced to illustrate that some properties of the measure-valued solution $ F^{n}_{t} $, especially for moments, will be maintained in the limiting process.  
\begin{lemma}\label{DissiLemma}
	Let $ F^{n}_{t} = \mathcal{F}^{-1}\left( \p^{n}(t,\xi) \right) $ and $ F_{t} = \mathcal{F}^{-1}\left( \p(t,\xi) \right) $ as defined in \eqref{weakconF} above. Then\\
	(i) For every $ t > 0 $, we have
	\begin{equation}\label{dissFt}
	\int_{\bR^{3}} |v|^{2} \,\rd F_{t}(v) \leq \int_{\bR^{3}} |v|^{2} \,\rd F_{0}(v).
	\end{equation}
	(ii) For any $  T >0 $ and $ \psi (v) \in C(\mathbb{R}^{3}) $ with $ |\psi(v) | \leq \left\langle v \right\rangle^{l}, 0 < l < 2 $, we have,
	\begin{equation}\label{limFt1}
	\lim\limits_{n\rightarrow \infty} \int_{\bR^{3}} \psi(v) \,\rd F_{t}^{n}(v) = \int_{\bR^{3}} \psi(v) \,\rd F_{t}(v)
	\end{equation} 
       uniformly for $ t\in [0,T] $.\\
       (iii) For any $  T >0 $ and $ \Psi (v,\vs) \in C(\mathbb{R}^{3}\times \mathbb{R}^{3}) $ with $ |\Psi(v,\vs) | \leq \left\langle v \right\rangle^{l} \left\langle \vs \right\rangle^{l}, 0 < l < 2 $, we have,
       \begin{equation}\label{limFt2}
       \lim\limits_{n\rightarrow \infty} \int_{\bR^{3}} \int_{\bR^{3}} \Psi (v,\vs) \,\rd F_{t}^{n}(v)\,\rd F_{t}^{n}(\vs) = \int_{\bR^{3}} \int_{\bR^{3}} \Psi (v,\vs) \,\rd F_{t}(v) \,\rd F_{t}(\vs)
       \end{equation} 
       uniformly for $ t\in [0,T] $.
\end{lemma}
\begin{proof}
	The proof in inelastic case is almost same as the counterpart \cite[Lemma 3.2]{morimoto2016measure} of the elastic case, since the estimates are only associated with the variables $ v $ and $ v_{*} $ in the limiting process, which have no distinction between the elastic and inelastic cases. See Appendix \ref{appen2} for the complete proof.
\end{proof}

\subsubsection{Proof of Theorem \ref{main1} (i)}

In this subsection, we are prepared to complete the proof of the main Theorem \ref{main1} (i) in the case of finite energy initial datum, based the preliminary results above.
\begin{proof}
\textit{Step 1:} Noticing that $ \nabla (\psi(v) - \psi(v_{*})) = O(|v-v_{*}|) $ for $ \psi \in C_{b}^{2}(\mathbb{R}^{3}) $, as well as the symmetric formula \eqref{sym1} on the spherical integration that,
\begin{equation}\label{sym2}
\int_{\Sd^{2}} b(\sigma \cdot \hat{q}) \left[  (\sigma\cdot\hat{q})\hat{q} -\sigma  \right] \,\rd \sigma = \lim\limits_{n\rightarrow \infty} \int_{\Sd^{2}} b_{n}(\sigma \cdot \hat{q}) \left[  (\sigma\cdot\hat{q})\hat{q} -\sigma  \right] \,\rd \sigma = 0
\end{equation}
and if we denote 
\begin{equation}\label{triangle}
\triangle \psi^{e}: = \psi(v'_{*}) + \psi(v') - \psi(v_{*}) - \psi(v) + a_{+}\frac{|v-v_{*}|}{2} \left[ \nabla\psi(v) - \nabla\psi(v_{*}) \right] \left[  (\sigma\cdot\hat{q})\hat{q} -\sigma  \right]
\end{equation}
it then follows from \eqref{fourpsi} that
\begin{equation}\label{trianglees}
\left| \triangle \psi^{e} \right|= O\left(|v-v_{*}|^{2} \theta^{2} \right),\ \text{or more precisely}, \    \left| \triangle \psi^{e} \right| \lesssim \sup_{j,k} \| \partial_{j} \partial_{k} \psi \|_{L^{\infty}} \left| v-v_{*} \right|^{2} \theta^{2}.
\end{equation}
Taking the symmetric property \eqref{sym1} and \eqref{sym2} on the spherical integration once again, we have
\begin{equation}\label{symmweak}
\begin{split}
\int_{\bR^{3}} \int_{\bR^{3}} \int_{\Sd^{2}} & b_{n}(\sigma \cdot \hat{q}) \left| v-v_{*}\right|^{\gamma} \phi_{n}(|v-v_{*}|) \left[ \psi(v'_{*}) + \psi(v') - \psi(v_{*}) - \psi(v)  \right] \,\rd \sigma \,\rd F^{n}_{t} (v) \,\rd F^{n}_{t} (v_{*}) \\
& - \int_{\bR^{3}} \int_{\bR^{3}} \int_{\Sd^{2}} b(\sigma \cdot \hat{q}) \left| v-v_{*}\right|^{\gamma}  \left[ \psi(v'_{*}) + \psi(v') - \psi(v_{*}) - \psi(v)  \right] \,\rd \sigma \,\rd F_{t} (v) \,\rd F_{t} (v_{*}) \\
= & \int_{\bR^{3}} \int_{\bR^{3}} \int_{\Sd^{2}} b_{n}(\sigma \cdot \hat{q}) \left| v-v_{*}\right|^{\gamma} \phi_{n}(|v-v_{*}|) \triangle \psi^{e} \,\rd \sigma \,\rd F^{n}_{t} (v) \,\rd F^{n}_{t} (v_{*}) \\
& - \int_{\bR^{3}} \int_{\bR^{3}} \int_{\Sd^{2}} b(\sigma \cdot \hat{q})  \left| v-v_{*}\right|^{\gamma}  \triangle \psi^{e} \,\rd \sigma \,\rd F_{t} (v) \,\rd F_{t} (v_{*}) \\
=& \underbrace{ \int_{\bR^{3}} \int_{\bR^{3}} \left[ \Psi^{e}_{n}(v,v_{*}) - \Psi^{e}(v,v_{*}) \right] \,\rd F^{n}_{t}(v) \,\rd F^{n}_{t}(v_{*}) }_{(I)} \\
& + \underbrace{ \left(  \int_{\bR^{3}} \int_{\bR^{3}} \Psi^{e}(v,v_{*})  \,\rd F^{n}_{t}(v) \,\rd F^{n}_{t}(v_{*}) - \int_{\bR^{3}} \int_{\bR^{3}} \Psi^{e}(v,v_{*})  \,\rd F_{t}(v) \,\rd F_{t}(v_{*}) \right) }_{(II)}
\end{split}
\end{equation}
where we denote $ B_{n} = b_{n}(\sigma \cdot \hat{q}) |v-v_{*}|^{\gamma} \phi_{n}(|v-v_{*}|)  $ and further $ \Psi^{e}_{n}(v,v_{*}), \Psi^{e}(v,v_{*}) $ as following, 
\begin{equation}
\Psi^{e}_{n}(v,v_{*}) = |v-v_{*}|^{\gamma} \phi_{n}(|v-v_{*}|) \int_{\Sd^{2}} b_{n}(\sigma \cdot \hat{q}) \triangle \psi^{e} \,\rd \sigma = L^{e}_{B_{n}}[\psi] (v,v_{*})
\end{equation}
and 
\begin{equation}
\Psi^{e}(v,v_{*}) = |v-v_{*}|^{\gamma} \int_{\Sd^{2}} b(\sigma \cdot \hat{q}) \triangle \psi^{e} \,\rd \sigma = L^{e}_{B}[\psi] (v,v_{*}).
\end{equation}
By using the elementary inequality $  |v-v_{*}|^{2 + \gamma} \leq \left\langle v \right\rangle^{2 + \gamma} \left\langle v_{*} \right\rangle^{2 + \gamma}  $ with $ -2 \leq \gamma < 0 $, and noting that both $ \Psi^{e}_{n}(v,v_{*}) $ and $ \Psi^{e}(v,v_{*}) $ satisfy the growth condition in Lemma \ref{DissiLemma}(iii) with $ l = 2 + \gamma < 2 $, then, for any $ R > 0 $, we have,
\begin{equation}
\left| \int_{|v|>R} \int_{|v_{*}|> R} \left[ \Psi^{e}_{n}(v,v_{*}) -  \Psi^{e}(v,v_{*})  \right] \,\rd F^{n}_{t}(v) \,\rd F^{n}_{t} (v_{*})  \right| \lesssim R^{l-2}\left( \int_{\bR^{3}} \left\langle v \right\rangle^{2} \,\rd F_{0}(v) \right)^{2}.
\end{equation}
Thus, the first term $ (I) $ of \eqref{symmweak} converges to zero uniformly for $ t \geq 0 $, as $ n\rightarrow \infty $, since the function $ \Psi^{e}_{n}(v,v_{*}) $ converges to $ \Psi^{e}(v,v_{*}) $ uniformly on a compact set of $ (v,v_{*}) \in  \mathbb{R}_{v}^{3} \times \mathbb{R}_{v_*}^{3} $; on the other hand, by Lemma \ref{DissiLemma} (iii), the second term $ (II) $ of \eqref{symmweak} also converges to zero uniformly on a compact set of $ t \in [0,\infty) $, as $ n \rightarrow \infty $. \\
Hence, it concludes that $ F_{t} $ satisfies the definition of measure-valued solution \eqref{weakmeasure}, \textit{i.e.}, the following mapping
\begin{equation}
t \in [0,\infty) \mapsto \int_{\bR^{3}} \int_{\bR^{3}} L^{e}_{B}[\psi](v,v_{*}) \,\rd F_{t}(v) \,\rd F_{t}(v_{*})
\end{equation}
lies in the space $ C[0,\infty)  \cap  L^{1}_{\text{loc}} [0,\infty) $, which is derived from the fact that $ F_{t}^{n} \in C \left(  [0,\infty); P_{2}(\mathbb{R}^{3}) \right) $, \textit{i.e.}, the mapping
\begin{equation}
t \mapsto \int_{\bR^{3}} \int_{\bR^{3}} L^{e}_{B_{n}}[\psi](v,v_{*}) \,\rd F^{n}_{t}(v) \,\rd F^{n}_{t}(v_{*})
\end{equation}
is continuous with respect to $t \in [0,\infty) $, in the sense of the Weierstrass Approximation Theorem as in the proof of \cite[Lemma 3.2 (iii)]{morimoto2016measure}.\\
\textit{Step 2:} What's more, the energy dissipation property \eqref{enerdiss} of the measure-valued solution $ F_{t}(v) $ has been shown in Lemma \ref{DissiLemma} (i); 
then the momentum conservation \eqref{momencon} can be proved by considering another cutoff function $  \psi_{m}(v) =  v_{j} \phi_{c}(|v| / m) \in C_{b}^{2}(\mathbb{R}^{3}), j=1,2,3 $ that, 
\begin{equation}\label{CC1}
\begin{split}
&\int_{\bR^{3}} \psi_{m}(v) \,\rd F_{t}(v) - \int_{\bR^{3}} \psi_{m}(v) \,\rd F_{0}(v) \\
= &\frac{1}{2} \int_{0}^{t} \left[ \int_{\bR^{3}} \int_{\bR^{3}} \int_{\Sd^{2}} b(\sigma \cdot \hat{q}) |v-v_*|^{\gamma} \triangle \psi^{e}_{m}(v) \,\rd F_{\tau}(v) \,\rd F_{\tau}(v_*) \,\rd \sigma   \right] \,\rd\tau < \infty
\end{split}
\end{equation}
where the finiteness with respect to $ m $ comes from the estimate \eqref{triangle}. Hence, for any $ \eta > 0 $, there always exists a $ \delta > 0 $ independent of $ m $ such that
\begin{equation}\label{CC2}
\left| \int \int_{|v-v_*| < \delta} \int_{\theta < \delta} b(\sigma \cdot \hat{q}) |v-v_*|^{\gamma} \triangle \psi^{e}_{m}(v) \,\rd F_{\tau}(v) \,\rd F_{\tau}(v_*) \,\rd \sigma \right| < \eta, 
\end{equation}
as a consequent, combining \eqref{CC1} together with \eqref{CC2}, the Lebesgue Dominated Convergence Theorem yields that
\begin{equation}
\lim\limits_{m\rightarrow \infty} \left| \int_{\bR^{3}} \psi_{m}(v) \,\rd F_{t}(v) - \int_{\bR^{3}} \psi_{m}(v) \,\rd F_{0}(v)  \right| \leq \frac{t\eta}{2},
\end{equation}
in fact, 
with the help of Lemma \ref{DissiLemma} (ii), we can also obtain momentum conservation \eqref{momencon} by the following estimate 
\begin{equation}
\int_{\bR^{3}} v_{j} \,\rd F_{t}(v) = \lim\limits_{n\rightarrow \infty} \int_{\bR^{3}} v_{j} \,\rd F^{n}_{t}(v) = \int_{\bR^{3}} v_{j} \,\rd F_{0}(v), \ j=1,2,3
\end{equation}
where we utilize \eqref{limFt1} in the first equality above, and \eqref{momentumcon} in the second equality.\\
\textit{Step 3:} Finally, it only remains to show that $ F_{t}(v) \in C\left([0,\infty); P_{2}(\mathbb{R}^{3})\right) $. In fact, considering the equi-continuous estimate \eqref{equitsoft} together with the Lemma \ref{DissiLemma}(i), it suffices to prove that $ F_{t}(v) \in C\left([0,\infty); P_{2}(\mathbb{R}^{3})\right) $, since it follows from \cite[Lemma 1]{ToscaniVillani1999} that, for any $ t_{0} \in [0,\infty) $, 
\begin{equation}
\lim\limits_{R \rightarrow \infty} \left( \lim\sup\limits_{t \rightarrow t_{0}} \int_{|v|\geq R} |v|^{2} \,\rd F_{t}(v) \right) = 0.
\end{equation}
This completes the proof of Theorem \ref{main1} (i) in the case of finite energy initial datum.
\end{proof}
Once the measure-valued solution $ F_{t} $ is constructed with finite energy estimate, the moment propagation Corollary \ref{momentspro} would immediately follow by using the another weighted function as the elastic case \cite[Proposition 1.5]{morimoto2016measure}, see Appendix \ref{appPro} for complete proof.

\subsection{Infinite Energy Initial Datum}
In this subsection, we will turn to the possibly infinite energy initial datum case, \textit{i.e.}, $ F_{0} \in P_{\alpha}(\mathbb{R}^{3}) $ with $ \alpha < 2 $.

\subsubsection{Preliminary}

We first give the following Proposition \ref{infinitePro}, which implies more strict moments estimation of the cutoff solution $ F_{t}^{n} $ with infinite energy initial datum $ F_{0} $.

\begin{proposition}\label{infinitePro}
	For any $ e \in (0,1]$, and assume that $ -2 \leq \gamma < 0 $ and $ c_{\gamma,s} < \alpha <2 $. If $ F^{n}_{t} \in C\left( [0,\infty), P_{\alpha'}(\mathbb{R}^{3}) \right)$ with $ 0 < \alpha' < \alpha $ is a measure-valued solution obtained in Theorem \ref{local} for the initial datum $ F_{0} \in P_{\alpha}(\mathbb{R}^{3}) $, then there exists a constant $ C_{e}(T) > 0 $ independent of $ n $ such that, for any $ T > 0 $,
	\begin{equation}\label{infinitees}
	\sup_{t \in [0,T]} \int_{\bR^{3}} \left\langle v \right\rangle^{\alpha} \,\rd F^{n}_{t}(v) \leq C_{e}(T) \int_{\bR^{3}} \left\langle v \right\rangle^{\alpha} \,\rd F_{0}(v).
	\end{equation}	
\end{proposition}
\begin{proof}
	By once again applying the weighted function $ W_{\delta} $ in the Corollary \ref{DissiLemma}, we have, for any fixed $ n $ and the solution $ F_{t}^{n} $ obtained in Theorem \ref{local},
	\begin{equation}
	\int_{\bR^{3}} \left\langle v \right\rangle^{\alpha} \,\rd F^{n}_{t}(v) < \infty.
	\end{equation}
	\textit{Step 1: For the $ 0 \leq s < 1/2 $ case.} We first simply denote $ \triangle \psi_{\alpha}: = \psi_{\alpha}(v'_{*}) + \psi_{\alpha}(v')  - \psi_{\alpha}(v_{*}) - \psi_{\alpha}(v) $ with $ \psi_{\alpha}(v) = \left\langle v \right\rangle^{\alpha} $, then by applying the mean value theorem to $ \psi_{\alpha}(v'_{*}) - \psi_{\alpha}(v_{*}) $ and $ \psi_{\alpha}(v') - \psi_{\alpha}(v) $, the following delicate estimate holds,
	\begin{equation}
	\left| \triangle \psi_{\alpha} \right| = O \left( a_{+} \theta |v-v_{*}|^{2} \right) = O \left( a_{+} \theta \left|v-v_{*}\right| \left[ \left\langle v \right\rangle^{(\alpha-1)^{+}} + \left\langle v_{*} \right\rangle^{(\alpha-1)^{+}} \right] \right)
	\end{equation}
	where $ \alpha^{+} := \max\{\alpha,0\} $. It further yields that, for any $ \delta \in (0,1] $,
	\begin{equation}
	\left| \triangle \psi_{\alpha} \right| \lesssim a^{\delta}_{+} \theta^{\delta} \left|v-v_{*}\right|^{\delta} \left( \left\langle v \right\rangle + \left\langle v_{*} \right\rangle \right)^{(1-\delta)\alpha + \delta (\alpha-1)^{+}}.
	\end{equation}
	Thus, the desired estimate \eqref{infinitees} can be directly derived by substituting $ \psi(v) = \left\langle v \right\rangle^{\alpha} $ into the definition of measure-valued solution \eqref{weakmeasure}, as long as $ \alpha > \max\{\frac{\gamma}{2s} + 1, 0\} $.\\
	\textit{Step 2: For the $ s \geq 1/2 $ case.} Letting $ \psi_{\alpha}(v) = \left\langle v \right\rangle^{\alpha} $ again, then we have
	\begin{equation}
	\begin{split}
	&\frac{\rd}{\rd t} \int_{\bR^{3}} \psi_{\alpha}(v) \,\rd F_{t}^{n}(v)\\
	=& \frac{1}{2} \int_{\bR^{3}}\int_{\bR^{3}} \int_{\Sd^{2}} B_{n} \triangle \psi^{e}_{\alpha}(v) \,\rd\sigma \,\rd F_{t}^{n}(v) \,\rd F_{t}^{n}(v_{*})\\
	=& \underbrace{ \frac{1}{2} \int \int_{|v-v_{*}|<1} \int_{\Sd^{2}} B_{n} \triangle \psi^{e}_{\alpha}(v) \,\rd\sigma \,\rd F_{t}^{n}(v) \,\rd F_{t}^{n}(v_{*}) }_{(I)}\\
	&\qquad \qquad \qquad \qquad \qquad \qquad  + \underbrace{ \frac{1}{2} \int \int_{|v-v_{*}| \geq 1} \int_{\Sd^{2}} B_{n} \triangle \psi^{e}_{\alpha}(v) \,\rd\sigma \,\rd F_{t}^{n}(v) \,\rd F_{t}^{n}(v_{*}) }_{(II)}
	\end{split}
	\end{equation}
	where $ B_{n} $ has the same definition of \eqref{symmweak}, and $ \triangle \psi^{e}_{\alpha} $ is similarly defined as $ \triangle \psi^{e} $ in \eqref{triangle} by selecting $ \psi(v) = \left\langle v \right\rangle^{\alpha}  $; since $ \left\langle v \right\rangle^{\alpha} \in C_{b}^{2}(\mathbb{R}^{3}) $, it yields the same estimate $ \left| \triangle \psi^{e}_{\alpha} \right| \lesssim  \left| v-v_{*} \right|^{2} \theta^{2} $ as \eqref{trianglees}. On the other hand, if applying the mean value theorem to $ \psi_{\alpha}(v'_{*}) - \psi_{\alpha}(v_{*}) $ and $ \psi_{\alpha}(v') - \psi_{\alpha}(v) $, we find, 
	\begin{equation}
	\left| \triangle \psi^{e}_{\alpha} \right| \lesssim \theta \left|v-v_{*}\right| \left( \left\langle v \right\rangle^{(\alpha-1)^{+}} + \left\langle v_{*} \right\rangle^{(\alpha-1)^{+}} \right)
	\end{equation}
	which yields another estimate of $ \left| \triangle \psi^{e}_{\alpha} \right| $ that, for any $ \delta \in (0,1] $,
	\begin{equation}\label{tri1}
	\left| \triangle \psi^{e}_{\alpha} \right| \lesssim \theta^{1+\delta} \left|v-v_{*}\right|^{1+\delta} \left( \left\langle v \right\rangle^{(\alpha-1)^{+}} + \left\langle v_{*} \right\rangle^{(\alpha-1)^{+}} \right)^{1-\delta}.
	\end{equation}
	(i) When $ \gamma + 2s < 1 $, we can choose an $ \epsilon > 0 $ such that $ \gamma+2s<1 -\epsilon $ if $ \alpha > 1 $ such that the desired estimate \eqref{infinitees} works from \eqref{tri1}; while, if $ \gamma+2s < \alpha \leq 1 $, we just choose another $ \epsilon > 0 $ such that $ \alpha > \gamma+2s + \epsilon $. In summary, we need to select $ \alpha > \max\{\gamma+2s, 0\} $ in this case.\\
	(ii) When $ \gamma + 2s \geq 1 $, it suffices to choose an $ \epsilon > 0 $ such that $ \alpha > \gamma / (2s -1 + \epsilon) > c_{\gamma, s} > 1 $, as well as an $ \delta = 2s-1+\epsilon $ in the estimate \eqref{tri1} for $ (II) $; on the other hand, by selecting $ \delta = 1 $ in the estimate \eqref{tri1} for $ (I) $, we are able to find that there exists a constant $ C_{e} > 0 $ independent of $ n $ such that 
	\begin{equation}
	\frac{\rd}{\rd t} \int_{\bR^{3}} \psi_{\alpha}(v) \,\rd F_{t}^{n}(v) \leq C_{e} \int_{\bR^{3}} \psi_{\alpha}(v) \,\rd F_{t}^{n}(v)
	\end{equation}
	which leads to estimate \eqref{infinitees} by directly applying Gronwall's inequality.
\end{proof}

\subsubsection{Proof of Theorem \ref{main1} (ii)}
In this subsection, the proof of the main Theorem \ref{main1} (ii) in the case of infinite energy initial datum is presented. In fact, the taking limit process can be guaranteed in the infinite energy case due to the same Proposition \ref{infinitePro}, and it only remains to show the limit function $ F_{t} \in C\left( [0,\infty); P_{\alpha}(\mathbb{R}^{3}) \right) $ is definitely the measure-valued solution.
\begin{proof}
Starting from the Proposition \ref{infinitePro} above, we find $ F_{t}^{n} \in P_{\alpha}(\mathbb{R}^{3}) $ with $ c_{\gamma,s} < \alpha <2 $; further noting the fact that $ v_{j}, j=1,2,3 $ is a collision invariant for any $ \alpha \geq 1 $. Then, similar to the elastic case \cite[(3.16)]{morimoto2016measure}, by recalling that $ \p^{n}(t, \xi) = \mathcal{F} \left( F_{t}^{n} \right) $ and the subsequence of $ \left\lbrace \p^{n}(t, \xi) \right\rbrace^{\infty}_{n=1}  $ converges to $ \p(t, \xi) = \mathcal{F} \left( F_{t} \right) $ uniformly in every compact subset of $ \mathbb{R} \times [0,\infty) $, it yields that, for any $ t\in (0,T] $, there exists a constant $ c_{\alpha} $ such that
\begin{equation}
\begin{split}
c_{\alpha} \int_{\bR^{3}} |v|^{\alpha} \,\rd F_{t}^{n}(v)  = & \lim\limits_{\delta \rightarrow 0^{+}} \int_{\delta \leq |\xi| \leq \delta^{-1}} \frac{\left| 1- \text{Re} \p(t, \xi)\right|}{|\xi|^{3 + \alpha}} \,\rd \xi\\
= & \lim\limits_{\delta \rightarrow 0^{+}} \lim\limits_{n \rightarrow \infty} \int_{\delta \leq |\xi| \leq \delta^{-1}} \frac{\left| 1- \text{Re} \p^{n}(t, \xi)\right|}{|\xi|^{3 + \alpha}} \,\rd \xi \\
\leq & C_{e}(T) c_{\alpha} \int_{\bR^{3}} \left\langle v \right\rangle^{\alpha} \,\rd F_{0}(v),
\end{split}
\end{equation}
from which we can find that Corollary \ref{DissiLemma} (i) still holds if replacing the weights' parameter $ 2 $  by $ c_{\gamma,s} < \alpha <2 $. Thus, for any $ 0 < \alpha' < \alpha $, we prove $ F_{t} \in C\left( [0,\infty); P_{\alpha'}(\mathbb{R}^{3}) \right) $ is a measure-valued solution to the equation \eqref{IB} -\eqref{F0} by the same steps in the finite energy case.\\
Furthermore, since $ F_{t} $ belongs to the space $ P_{\alpha}(\mathbb{R}^{3}) $ uniformly with respect to $ t $ in any finite time interval of $ t\in [0,\infty) $, we then obtain that
\begin{equation}
\left| \int_{\bR^{3}} \left\langle v \right\rangle^{\alpha} \,\rd F_{t}(v) - \int_{\bR^{3}} \left\langle v \right\rangle^{\alpha} \,\rd F_{s}(v)   \right| = O\left( |t-s| \right)
\end{equation}
by applying the same estimation for $ \triangle \psi_{\alpha} $ and $ \triangle \psi^{e}_{\alpha} $ used in Proposition \ref{infinitePro} above, which implies that $ F_{t} \in C\left( [0,\infty); P_{\alpha}(\mathbb{R}^{3}) \right) $. And this completes the proof of Theorem \ref{main1} (ii) in the case of infinite energy initial datum.
\end{proof}



\begin{appendices}
\setcounter{equation}{0}
\renewcommand{\theequation}{A.\arabic{equation}}
\renewcommand{\thesubsection}{A.\arabic{subsection}}	
	
	\subsection{Fourier Transformation of the inelastic Boltzmann operator with soft potential}\label{appen1}
	We take the gain term $ Q^{+}_{e}(g,f) $ as example, since the loss term $ Q^{-}_{e}(g,f) $  is the same as the elastic case. Considering the pre-post velocity relation \eqref{ve} in inelastic case, we have
	\begin{equation}
	\begin{split}
	&\mathcal{F}\left[ Q^{+}_{e}(g,f) \right](\xi) \\
	=& \int_{\bR^{6}} \int_{\Sd^{2}} g(\vs) f(v) b\left(\frac{v-\vs}{|v-\vs|}\cdot \sigma\right) \Phi_{c}\left(|v-v_{*}|\right) \e^{ -i \left(\frac{v+\vs}{2} + \frac{1-e}{4}(v-\vs) + \frac{1+e}{4}|v-\vs|\sigma\right)\cdot \xi}  \rd \sigma \rd \vs \rd v\\
	=& \int_{\bR^{6}} \int_{\Sd^{2}} g(\vs) f(v) b\left(\frac{v-\vs}{|v-\vs|}\cdot \sigma\right)\Phi_{c}\left(|v-v_{*}|\right) \e^{ -i\frac{v+\vs}{2}\cdot\xi}\e^{-i \left(\frac{1-e}{4}(v-\vs) + \frac{1+e}{4}|v-\vs|\sigma\right)\cdot\xi} \rd \sigma \rd \vs \rd v 
	\end{split}
	\end{equation}
	according to the general change of variable,
	\begin{equation}
	\int_{\Sd^{2}} F(k\cdot\sigma, l\cdot\sigma) \rd \sigma = \int_{\Sd^{2}} F(l\cdot\sigma, k\cdot\sigma) \rd \sigma, \ |l| = |k| = 1 
	\end{equation}
	due to the existence of an isometry on $ \Sd^{2} $ exchanging $ l $ and $ k $, we have, by exchanging the rule of $ \frac{\xi}{|\xi|} $ and $ \frac{v-\vs}{|v-\vs|} $
	\begin{equation}
	\begin{split}
	\int_{\Sd^{2}} g(\vs) f(v) b\left(\frac{v-\vs}{|v-\vs|}\cdot \sigma\right)\Phi_{c}\left(|v-v_{*}|\right) \e^{-i \left(\frac{1-e}{4}(v-\vs) + \frac{1+e}{4}|v-\vs|\sigma\right)\cdot\xi}  \rd \sigma \\
	= \int_{\Sd^{2}} g(\vs) f(v) b\left(\frac{\xi}{|\xi|}\cdot \sigma\right) \int_{\bR^{3}} \hat{\Phi}_{c}\left(\zeta\right) \e^{i\zeta\cdot(v-\vs)} \e^{-i \left(\frac{1-e}{4}\xi + \frac{1+e}{4}|\xi|\sigma\right)\cdot(v-\vs)} \rd\zeta \rd \sigma
	\end{split}
	\end{equation}
	where the $ \hat{\Phi}_{c}\left(\zeta\right) = \mathcal{F}\left[ \Phi_{c} \right] $. Thus, 
	\begin{equation}
	\begin{split}
	&\mathcal{F}\left[ Q^{+}_{e}(g,f) \right](\xi) \\
	=&\int_{\bR^{6}} \int_{\Sd^{2}} g(\vs) f(v) b\left(\frac{v-\vs}{|v-\vs|}\cdot \sigma\right) \Phi_{c}\left(|v-v_{*}|\right) \e^{ -i\frac{v+\vs}{2}\cdot\xi}\e^{-i \left(\frac{1-e}{4}(v-\vs) + \frac{1+e}{4}|v-\vs|\sigma\right)\cdot\xi} \rd \sigma \rd \vs \rd v \\
	=& \int_{\bR^{6}} \int_{\Sd^{2}} b\left(\frac{\xi}{|\xi|}\cdot \sigma\right) \int_{\bR^{3}} \hat{\Phi}_{c}\left(\zeta\right) \e^{i\zeta\cdot(v-\vs)}  g(\vs) f(v)  \e^{ -i\frac{v+\vs}{2}\cdot\xi}
	\e^{-i \left(\frac{1-e}{4}\xi + \frac{1+e}{4}|\xi|\sigma\right)\cdot(v-\vs)} \rd\zeta \rd \sigma \rd \vs \rd v\\
	=& \int_{\bR^{6}}\int_{\Sd^{2}} b\left(\frac{\xi}{|\xi|}\cdot \sigma\right)\int_{\bR^{3}} \hat{\Phi}_{c}(\zeta) g(\vs) f(v)  \e^{ -iv\cdot\left( \frac{\xi}{2} + \frac{1-e}{4}\xi + \frac{1+e}{4}|\xi|\sigma - \zeta \right)} \e^{-i \vs \cdot\left( \frac{\xi}{2} - \frac{1-e}{4}\xi - \frac{1+e}{4}|\xi|\sigma + \zeta \right) } \rd\zeta \rd \sigma \rd \vs \rd v\\
	=& \int_{\Sd^{2}} b\left(\frac{\xi}{|\xi|}\cdot \sigma\right) \int_{\bR^{3}} \hat{\Phi}_{c}(\zeta) \hat{f}(\xie^{+} - \zeta) \hat{g}(\xie^{-}+ \zeta)  \rd\zeta \rd \sigma
	\end{split}
	\end{equation}
	where, unlike the elastic case, the $ \xi^{+} $ and $ \xi^{-} $ are defined as
	\begin{equation}
	\xie^{+} = \frac{\xi}{2} + \frac{1-e}{4}\xi + \frac{1+e}{4}|\xi|\sigma, \quad
	\xie^{-} = \frac{\xi}{2} - \frac{1-e}{4}\xi - \frac{1+e}{4}|\xi|\sigma.
	\end{equation}

	\subsection{Proof of Lemma \ref{DissiLemma}}\label{appen2}
	\begin{proof}
	For (i), considering the definition of the Fourier transform of the probability measure $ F^{n}_{t} $, which is the measure-valued solution to the cutoff equation \eqref{eqpcuto} obtained by Theorem \ref{local}, that
	\begin{equation}
	\frac{\p^{n}(\xi) - 1}{|\xi|^{\alpha}} \leq \int_{\bR^{3}} \frac{|\e^{-iv\cdot\xi} - 1|}{|\xi|^{\alpha}} \,\rd F^{n}_{t}(v), 
	\end{equation}
	together with \cite[Lemma 3.15]{cannone2010infinite}, we have
	\begin{equation}
	\|\p^{n} - 1\|_{2} \lesssim \int_{\bR^{3}} |v|^{2}\,\rd F^{n}_{t}(v)
	\end{equation}
	as a consequent, for any $ \delta \in (0,1) $, it follows that
	\begin{equation}
	\sup_{\delta < |\xi| < \delta^{-1}} \frac{|\p^{n}(\xi) - 1|}{|\xi|^{2}} \leq \|\p^{n} - 1\|_{2} \lesssim \int_{\bR^{3}} |v|^{2}\,\rd F^{n}_{t}(v) \leq \int_{\bR^{3}} |v|^{2}\,\rd F_{0}(v).
	\end{equation}
	where we utilize the \eqref{energydiss} in the last inequality above.
	By letting $ \delta \rightarrow 0 $, we then obtain,
	\begin{equation}
	\|\p - 1\|_{2} = \lim\limits_{\delta \rightarrow 0} \sup_{\delta < |\xi| < \delta^{-1}} \frac{|\p - 1|}{|\xi|^{2}} \leq \lim\limits_{\delta \rightarrow 0} \lim\limits_{n\rightarrow \infty} \sup_{\delta < |\xi| < \delta^{-1}} \frac{|\p^{n} - 1|}{|\xi|^{2}} \lesssim \int_{\bR^{3}} |v|^{2}\,\rd F_{0}(v).
	\end{equation}
	Due to the fact that $ \mathcal{K}^{2} = \mathcal{F} \left( P_{2}(\mathbb{R}^{3}) \right) $, it suffices to prove that $ \int_{\bR^{3}} |v|^{2}\,\rd F_{t}(v) < \infty $ for any $ t>0 $, $ i.e. $, for any $ \epsilon > 0 $, we are able to find a $ R_{t, \epsilon} > 0 $ large enough such that
	\begin{equation}
	\int_{|v| \geq R_{t,\epsilon}} |v|^{2} \,\rd F_{t}(v) < \epsilon.
	\end{equation}
	With help of the convergent result \eqref{weakconF}, we can further obtain
	\begin{equation}
	\int_{\bR^{3}} |v|^{2} \,\rd F_{t}(v) \leq \lim\limits_{n\rightarrow \infty} \int_{\bR^{3}} |v|^{2} \mathbf{1}_{|v|\leq R_{t, \epsilon}} \,\rd F^{n}_{t}(v) + \epsilon \leq \int_{\bR^{3}} |v|^{2}\,\rd F_{0}(v) + \epsilon.
	\end{equation}
	Finally we can prove \eqref{dissFt} by letting $ \epsilon \rightarrow 0 $.
	
	For (ii), for any test function $ \psi (v) \in C(\mathbb{R}^{3}) $ satisfying the growing condition $ |\psi(v) | \leq \left\langle v \right\rangle^{l}$ with $ 0 < l < 2 $, we have, for any $ R > 0 $,
	\begin{equation}
	R^{2-l} \int_{|v|\geq R} |\psi(v)| \,\rd F_{t}^{n}(v) \lesssim \int_{|v|\geq R} \left\langle v \right\rangle^{2}  \,\rd F_{t}^{n}(v) \leq \int_{\bR^{3}} \left\langle v \right\rangle^{2}  \,\rd F_{t}^{n}(v) \leq \int_{\bR^{3}} \left\langle v \right\rangle^{2}  \,\rd F_{0}(v)
	\end{equation}
	where the the limit $ F_{t}(v) $ satisfies the estimate above as well.
	Then, for any $ \epsilon > 0 $, we can find $ R_{1}(\epsilon) > 0 $ and $ R_{2}(\epsilon) > 0 $ big enough such that
	\begin{equation}
	\int_{|v|\geq R_{1}(\epsilon)} |\psi(v)| \,\rd F_{t}^{n}(v) \leq \epsilon \quad \text{and} \quad \int_{|v|\geq R_{2}(\epsilon)} |\psi(v)| \,\rd F_{t}(v) \leq \epsilon
	\end{equation}
	Hence, after selecting $ R_{\epsilon} = \max\{ R_{1}(\epsilon), R_{2}(\epsilon) \} $, we can obtain that
	\begin{equation}
	\int_{|v|\geq R_{1}(\epsilon)} |\psi(v)| \,\rd F_{t}^{n}(v) + \int_{|v|\geq R_{2}(\epsilon)} |\psi(v)| \,\rd F_{t}(v) \leq 2\epsilon.
	\end{equation}
	Let $ \chi_{R_{\epsilon}} (v) \in C_{c}^{\infty}(\mathbb{R}^{3}) $ be a indicator function that
	\begin{equation}\label{chiR}
	\chi_{R_{\epsilon}} (v) := \begin{cases}
	1,& \text{if}\ |v| \leq R_{\epsilon} \\
	0,& \text{if}\ |v| > R_{\epsilon}
	\end{cases}
	\end{equation}
	and $ \psi_{\epsilon} \in C_{c}^{\infty}(\mathbb{R}^{3}) $ be such that
	\begin{equation}
	\| \psi_{\epsilon} - \chi_{R_{\epsilon}} \psi_{\epsilon} \| < \epsilon
	\end{equation}
	Considering that $ \widehat{\psi_{\epsilon}} = \mathcal{F}(\psi_{\epsilon}) \in \mathcal{S}(\mathbb{R}^{3}) $, where $ \mathcal{S} $ is the classical Schwartz space, we have 
	\begin{equation}
	\begin{split}
	\left| \int_{\bR^{3}} \chi_{R_{\epsilon}} \psi_{\epsilon} \,\rd F^{n}_{t} - \int_{\bR^{3}} \chi_{R_{\epsilon}} \psi_{\epsilon} \,\rd F_{t}  \right| \leq & \left| \int_{\bR^{3}} \psi_{\epsilon} \,\rd F^{n}_{t} - \int_{\bR^{3}} \psi_{\epsilon} \,\rd F_{t}  \right| + 2\epsilon \\
	= & \left(2\pi\right)^{3} \left| \int_{\bR^{3}} \overline{\widehat{\psi_{\epsilon}}(\xi)} \left[  \p^{n}(t,\xi) - \p(t,\xi)  \right] \right| + 2\epsilon\\
	\lesssim & \frac{1}{M} \int_{|\xi|\geq M} \left\langle \xi \right\rangle \overline{\widehat{\psi_{\epsilon}}(\xi)} \,\rd \xi + \sup_{\xi\leq M} \left| \p^{n}(t,\xi) - \p(t,\xi) \right| + \epsilon
	\end{split}
	\end{equation}
	for any $ M >0 $. This completes the proof of \eqref{limFt1}.
	
	For (iii), we denote, for any $ R >0 $, 
	\begin{equation}
	K_{R} = \left\lbrace (v,v_{*}) : |v|\leq R, |v_{*}|\leq R \right\rbrace 
	\end{equation}
	which is a compact set of $ \mathbb{R}^{3} \times \mathbb{R}^{3} $ and $ K^{c}_{R} $ is the complement set of $ K_{R} $. Then, we have the following estimate of the cutoff solution $ F_{t}^{n} $ in $ K^{c}_{R} $,
	\begin{equation}
	\begin{split}
	\int\int_{K^{c}_{R}} \left| \Psi(v,v_{*}) \right| \,\rd F_{t}^{n}(v) \,\rd F_{t}^{n}(v_{*}) \lesssim & \frac{1}{\left\langle R \right\rangle^{2-l} } \int\int_{K^{c}_{R}} \left\langle v \right\rangle^{2} \left\langle v_{*} \right\rangle^{2}  \,\rd F_{t}^{n}(v) \,\rd F_{t}^{n}(v_{*})\\
	\leq  & \frac{1}{\left\langle R \right\rangle^{2-l} } \left(  \int_{\bR^{3}} \left\langle v \right\rangle^{2} \,\rd F_{t}^{n}(v)  \right)^{2}\\
	\leq  & \frac{1}{\left\langle R \right\rangle^{2-l} } \left(  \int_{\bR^{3}} \left\langle v \right\rangle^{2} \,\rd F_{0}(v)  \right)^{2}
	\end{split}
	\end{equation}
	similar to the technique in the proof of (ii) above, it follows that, we can find $ R_{\epsilon} > 0 $ large enough, such that
	\begin{equation}
	\int\int_{K^{c}_{R_{\epsilon}}} \left| \Psi(v,v_{*}) \right| \,\rd F_{t}^{n}(v) \,\rd F_{t}^{n}(v_{*}) + \int\int_{K^{c}_{R_{\epsilon}}} \left| \Psi(v,v_{*}) \right| \,\rd F_{t}(v) \,\rd F_{t}(v_{*}) \leq \epsilon .
	\end{equation}
	As a result, to prove \eqref{limFt2}, it remains to show that, in $ K^{c}_{R} $,
	\begin{equation}
	\lim\limits_{n\rightarrow \infty} \int\int_{ K^{c}_{R} } \Psi (v,\vs) \,\rd F_{t}^{n}(v)\,\rd F_{t}^{n}(\vs) = \int\int_{ K^{c}_{R} } \Psi (v,\vs) \,\rd F_{t}(v) \,\rd F_{t}(\vs)
	\end{equation}
	Recalling the Weierstrass Approximation Theorem applied in \cite[Lemma 3.2(iii)]{morimoto2016measure}, there exists a polynomial function $ \sum_{j=1}^{N} p_{j}(v) \tilde{p}_{j}(v_{*}) $ such that 
	\begin{equation}\label{Weier}
	\sup_{(v,v_{*}) \in K_{R_{\epsilon}}} \left| \Psi(v,v_{*}) - \sum_{j=1}^{N(\epsilon)} p_{j}(v) \tilde{p}_{j}(v_{*}) \right| < \epsilon.
	\end{equation}
	Taking advantage of indicator function $ \chi_{R_{\epsilon}} $ as \eqref{chiR} once again, we have, for each $ p_{j}(v) $ and $ \tilde{p}_{j}(v_{*}) $,
	\begin{equation}
	\lim\limits_{n\rightarrow \infty} \int_{\bR^{3}} p_{j}(v) \chi_{R_{\epsilon}} \,\rd F^{n}_{t}(v) = \int_{\bR^{3}} p_{j}(v) \chi_{R_{\epsilon}} \,\rd F_{t}(v)
	\end{equation}
	and 
	\begin{equation}
	\lim\limits_{n\rightarrow \infty} \int_{\bR^{3}} \tilde{p}_{j}(v_{*}) \chi_{R_{\epsilon}} \,\rd F^{n}_{t}(v_{*}) = \int_{\bR^{3}} \tilde{p}_{j}(v_{*}) \chi_{R_{\epsilon}} \,\rd F_{t}(v_{*})
	\end{equation}
	uniformly for $ t \in [0,T] $, \textit{i.e.}, for any $ \epsilon > 0 $, we can find $ N_{j,\epsilon} $ large enough, such that for all $ n > N_{j,\epsilon} $
	\begin{equation}\label{pn1}
	\left| \int_{|v|\leq R_{\epsilon}} p_{j}(v) \,\rd F^{n}_{t}(v) - \int_{|v|\leq R_{\epsilon}} p_{j}(v) \,\rd F_{t}(v) \right| \leq \frac{\epsilon}{\alpha_{j}N(\epsilon)}
	\end{equation}
	where $ \alpha_{j} = \sup_{|v|\leq R_{\epsilon}} \left| p_{j}(v)\right| $; similarly, we can also find another  $ \tilde{N}_{j,\epsilon} $ large enough, such that for all $ n > \tilde{N}_{j,\epsilon} $
	\begin{equation}\label{pn2}
	\left| \int_{|v_{*}|\leq R_{\epsilon}} \tilde{p}_{j}(v_{*}) \,\rd F^{n}_{t}(v_{*}) - \int_{|v_{*}|\leq R_{\epsilon}}  \tilde{p}_{j}(v_{*}) \,\rd F_{t}(v_{*}) \right| \leq \frac{\epsilon}{\tilde{\alpha}_{j}N(\epsilon)}
	\end{equation}
	where $ \tilde{\alpha}_{j} = \sup_{|v_{*}|\leq R_{\epsilon}} \left| \tilde{p}_{j}(v_{*})\right| $.\\
	Thus, if selecting $ N = \max_{1 \leq j \leq N(\epsilon)} \{ N_{j,\epsilon}, \tilde{N}_{j,\epsilon}  \} $ together with the estimates \eqref{Weier}, \eqref{pn1} and \eqref{pn2}, it yields that, for all $ n > N  $,
	\begin{equation}
	\begin{split}
	&\left| \int\int_{K_{R_{\epsilon}}} \Psi(v,v_{*}) \,\rd F^{n}_{t}(v) \,\rd F^{n}_{t}(v_{*}) - \int\int_{K_{R_{\epsilon}}} \Psi(v,v_{*}) \,\rd F_{t}(v) \,\rd F_{t}(v_{*})  \right| \\
	\lesssim & \sum_{j=1}^{N(\epsilon)} \left| \int\int_{K_{R_{\epsilon}}} p_{j}(v) \tilde{p}_{j}(v_{*}) \,\rd F^{n}_{t}(v) \,\rd F^{n}_{t}(v_{*}) - \int\int_{K_{R_{\epsilon}}} p_{j}(v) \tilde{p}_{j}(v_{*})  \,\rd F_{t}(v) \,\rd F_{t}(v_{*})  \right| + 2\epsilon \\
	\lesssim & \sum_{j=1}^{N(\epsilon)} \left| \int_{|v\leq R_{\epsilon}|}  p_{j}(v) \,\rd F^{n}_{t}(v) \left( \int_{|v_{*}|\leq R_{\epsilon}} \tilde{p}_{j}(v_{*}) \,\rd F^{n}_{t}(v_{*}) - \int_{|v_{*}|\leq R_{\epsilon}} \tilde{p}_{j}(v_{*}) \,\rd F^{n}_{t}(v_{*})  \right) \right|\\
	& + \sum_{j=1}^{N(\epsilon)} \left|  \left( \int_{|v|\leq R_{\epsilon}} p_{j}(v) \,\rd F^{n}_{t}(v) - \int_{|v|\leq R_{\epsilon}} p_{j}(v) \,\rd F_{t}(v)  \right) \int_{|v_{*} \leq R_{\epsilon}|}  \tilde{p}_{j}(v_{*}) \,\rd F_{t}(v_{*}) \right| + 2\epsilon\\
	\lesssim & 3\epsilon
	\end{split}
	\end{equation}
	By letting $ \epsilon \rightarrow 0 $, this will complete the proof of \eqref{limFt2}.
       \end{proof}

       \subsection{Proof of Corollary \ref{momentspro}}\label{appPro}
       \begin{proof}
       	\textit{Step 1: For the case of $ 0 < \kappa \leq 2 $.} We apply the similar strategy as the proof of cutoff solution in Corollary \ref{DissCutoff}. By introducing the weighted function:
       	\begin{equation}
       	\mathcal{W}_{\delta,\kappa}(x) = \frac{(1+x)^{1+\kappa/2}}{1+\delta(1+x)^{\kappa/2}},\quad \delta > 0,
       	\end{equation}
       	together with the dissipation of energy \eqref{dissFt} in Lemma \ref{DissiLemma}, we obtain that
       	\begin{equation}
       	\int_{\bR^{3}} \mathcal{W}_{\delta,\kappa}(|v|^{2}) \,\rd F_{t} < \infty.
       	\end{equation}
       	Furthermore, as the proof of cutoff case, we need to show the convexity of the weighted function $ \mathcal{W}_{\delta} $. Starting from the general form: for all $ n\in \mathbb{N} $ and $ n \geq 1 $,
       	\begin{equation}
       	\mathcal{W}_{\delta,\kappa,n}(x) = \frac{(1+x)^{1+(n\kappa)/2}}{1+\delta(1+x)^{\kappa/2}}
       	\end{equation}
       	it follows from the direct calculation that
       	\begin{equation}
       	\mathcal{W}''_{\delta,\kappa,n}(x) = \frac{(1+x)^{n\kappa/2-1}K}{\left[ 1+\delta(1+x)^{\kappa/2}\right]^{3}}
       	\end{equation}
       	where $ K $ is denoted by 
       	\begin{equation}
       	\begin{split}
       	K :=& \left(1+\frac{n\kappa}{2}\right) \left(\frac{n\kappa}{2} + \left[\delta(1+x)^{\kappa/2}\right] \frac{(n-2)\kappa}{2} \right) \\
       	&+ \left( 1 +\frac{(n-1)\kappa}{2} \right) \left( \frac{(n+1)\kappa}{2} + \frac{(n-1)\kappa}{2}\left[\delta(1+x)^{\kappa/2}\right] \right)\left[\delta(1+x)^{\kappa/2}\right].
       	\end{split}
       	\end{equation} 
       	Then, when $ n=1 $,
       	\begin{equation}
       	0 \leq K = \frac{\kappa}{2}\left(1 + \frac{\kappa}{2} \right) + \frac{\kappa}{2}\left(1 - \frac{\kappa}{2} \right) \left[\delta(1+x)^{\kappa/2}\right] \leq \frac{\kappa}{2}\left(1 + \frac{\kappa}{2} \right) \left[1+\delta(1+x)^{\kappa/2}\right]
       	\end{equation}
       	 when $ n\geq 2 $,
       	 \begin{equation}
       	 \begin{split}
       	 0 \leq K = &\frac{n\kappa}{2}\left(1 + \frac{n\kappa}{2} \right) \left[1+\delta(1+x)^{\kappa/2}\right] \\
       	 &+ \frac{(n+1)\kappa}{2}\left(1 + \frac{(n-1)\kappa}{2} \right)  \left[\delta(1+x)^{\kappa/2}\right] \left[1+\delta(1+x)^{\kappa/2}\right]\\
       	  \leq& \frac{(n+1)\kappa}{2}\left(1 + \frac{(n-1)\kappa}{2} \right)  \left[\delta(1+x)^{\kappa/2}\right].
       	 \end{split}
       	 \end{equation}
       	 Therefore, we obtain the bounded estimate for the second-order derivative the weighted function $ \mathcal{W}_{\delta,\kappa,n}(x) $, that is 
       	 \begin{equation}
       	 0 \leq \mathcal{W}''_{\delta,\kappa,n}(x) \leq C_{n.\kappa} (1+x)^{n\kappa/2 - 1}
       	 \end{equation}
       	 where the upper bound is uniform for $ \delta $.
       	 As a result, by selecting $ \mathcal{W}_{\delta,\kappa}(x) = \mathcal{W}_{\delta,\kappa,1}(x) $, we can get the similar estimate as the cutoff case, that is 
       	 \begin{equation}
       	 \begin{split}
       	 \int_{\bR^{3}} \mathcal{W}_{\delta,\kappa}(|v|^{2}) \,\rd F_{t}(v) \leq& \int_{\bR^{3}} \mathcal{W}_{\delta,\kappa}(|v|^{2}) \,\rd F_{0}(v) + C \int_{\bR^{3}} \left\langle v \right\rangle^{2} \,\rd F_{0}(v) \int_{0}^{t}  \int_{\bR^{3}} \left\langle v \right\rangle^{2} \,\rd F_{\tau}(v) \,\rd \tau\\
       	 \leq & C_{1} + C_{2}t
       	 \end{split}
       	 \end{equation} 
       	 Finally, by letting $ \delta \rightarrow 0 $, we obtain that, for any $ t \in [0,T] $, the following estimate holds:
       	 \begin{equation}
       	 \int_{\bR^{3}} \left\langle v \right\rangle^{2+ \kappa} \,\rd F_{t}(v) \leq C_{1} + C_{2} T.
       	 \end{equation}
       	 \textit{Step 2: For the case of $ \kappa > 2 $.} In this case, we can always find a $ \kappa' \in (0,2) $ and an integer $ N $ such that $ \kappa' N = \kappa $. Then it follows that, for all $ j=1,...,N $,
       	 \begin{equation}
       	 \int_{\bR^{3}} \left\langle v \right\rangle^{2+ j \kappa'} \,\rd F_{0}(v) \leq \int_{\bR^{3}} \left\langle v \right\rangle^{2+ \kappa} \,\rd F_{0}(v)
       	 \end{equation}
       	 Now we are able to take advantage of the argument in Step 1 above to $ \mathcal{W}_{\delta,\kappa',1}(|v|^{2}) $, thus, by letting $ \delta \rightarrow 0 $, we have, for all $ t \in [0,T] $,
       	 \begin{equation}
       	 \int_{\bR^{3}} \left\langle v \right\rangle^{2+\kappa'} \,\rd F_{t}(v) \leq C_{1} + C_{2}T,
       	 \end{equation}
       	 which implies that, 
       	 \begin{equation}
       	 \int_{\bR^{3}} \mathcal{W}_{\delta,\kappa',2}(|v|^{2}) \,\rd F_{t}(v) \leq \infty, \ \text{for all}\ t\in [0,T].
       	 \end{equation}
       	 Moreover, by letting $ 0 < \kappa' < m\kappa' < 2 < (m+1)\kappa' < N\kappa' = \kappa $, we can apply the step 1 above once again to $ \mathcal{W}_{\delta,\kappa',j}(|v|^{2}), j = 1,...,m $ to obtain that, for all $ t \in [0,T] $,
       	 \begin{equation}
       	 \int_{\bR^{3}} \left\langle v \right\rangle^{2+m\kappa'} \,\rd F_{t}(v) \leq C_{1} + C_{2}T.
       	 \end{equation}
       	 Due to the fact that $ (m+1)\kappa' > 2 $, we can further deduce that
       	 \begin{equation}
       	 \begin{split}
       	 &\int_{\bR^{3}} \mathcal{W}_{\delta,\kappa',m+1}(|v|^{2}) \,\rd F_{t}(v)\\
       	 \leq& \int_{\bR^{3}} \mathcal{W}_{\delta,\kappa',m+1}(|v|^{2}) \,\rd F_{0}(v)
       	 + C \int_{\bR^{3}} \left\langle v \right\rangle^{2} \,\rd F_{0}(v) \int_{0}^{t} \int_{\bR^{3}} \left\langle v \right\rangle^{(m+1)\kappa'} \,\rd F_{\tau}(v) \,\rd \tau\\
       	 \leq & \int_{\bR^{3}} \left\langle v \right\rangle^{2+\kappa} \,\rd F_{0}(v) + C \int_{\bR^{3}} \left\langle v \right\rangle^{2} \,\rd F_{0}(v) \int_{0}^{t} \int_{\bR^{3}} \left\langle v \right\rangle^{2+m\kappa'} \,\rd F_{\tau}(v) \,\rd \tau\\
       	 \lesssim & \int_{\bR^{3}} \left\langle v \right\rangle^{2+\kappa} \,\rd F_{0}(v) + C_{\kappa}(t)
       	 \end{split}
       	 \end{equation}
       	 By continuing the estimate above, we finally obtain that, for any $ \kappa > 2 $,
       	 \begin{equation}
       	 \int_{\bR^{3}} \left\langle v \right\rangle^{2+\kappa} \,\rd F_{t}(v) \lesssim \int_{\bR^{3}} \left\langle v \right\rangle^{2+\kappa} \,\rd F_{0}(v) + C_{\kappa}(t)
       	 \end{equation}
       	 where $ C_{\kappa}(t) $ is the generic polynomial function of $ t $ arsing from the repeated calculations. 
       \end{proof}
\end{appendices}

\section*{Acknowledgement}
\label{sec:ack}
The author would like to thank
Professor Tong Yang for his introducing and discussing the related topic.

\bibliographystyle{plain}
\bibliography{Bib_Kunlun_inelastic}

\begin{thebibliography}{10}

\bibitem{AL2014CMP}
R.~Alonso and B.~Lods.
\newblock Boltzmann model for viscoelastic particles: asymptotic behavior,
  pointwise lower bounds and regularity.
\newblock {\em Comm. Math. Phys.}, 331(2):545--591, 2014.

\bibitem{Alonso2009}
R.~J. Alonso.
\newblock Existence of global solutions to the {C}auchy problem for the
  inelastic {B}oltzmann equation with near-vacuum data.
\newblock {\em Indiana Univ. Math. J.}, 58(3):999--1022, 2009.

\bibitem{ABCL2018}
R.~J. Alonso, V.~Bagland, Y.~Cheng, and B.~Lods.
\newblock One-dimensional dissipative {B}oltzmann equation: measure solutions,
  cooling rate, and self-similar profile.
\newblock {\em SIAM J. Math. Anal.}, 50(1):1278--1321, 2018.

\bibitem{AL2010}
R.~J. Alonso and B.~Lods.
\newblock Free cooling and high-energy tails of granular gases with variable
  restitution coefficient.
\newblock {\em SIAM J. Math. Anal.}, 42(6):2499--2538, 2010.

\bibitem{BCT2006}
M.~Bisi, J.~A. Carrillo, and G.~Toscani.
\newblock Decay rates in probability metrics towards homogeneous cooling states
  for the inelastic {M}axwell model.
\newblock {\em J. Stat. Phys.}, 124(2-4):625--653, 2006.

\bibitem{Bobylev2000inelastic}
A.~V. Bobylev, J.~A. Carrillo, and I.~M. Gamba.
\newblock On some properties of kinetic and hydrodynamic equations for
  inelastic interactions.
\newblock {\em J. Stat. Phys.}, 98(3-4):743--773, 2000.

\bibitem{Brilliantov2004}
N.~V. Brilliantov and T.~P\"{o}schel.
\newblock {\em Kinetic theory of granular gases}.
\newblock Oxford Graduate Texts. Oxford University Press, Oxford, 2004.

\bibitem{cannone2010infinite}
M.~Cannone and G.~Karch.
\newblock Infinite energy solutions to the homogeneous {B}oltzmann equation.
\newblock {\em Comm. Pure Appl. Math.}, 63(6):747--778, 2010.

\bibitem{CMWY2016}
Y.~K. Cho, Y.~Morimoto, S.~Wang, and T.~Yang.
\newblock Probability measures with finite moments and the homogeneous
  {B}oltzmann equation.
\newblock {\em SIAM J. Math. Anal.}, 48(4):2399--2413, 2016.

\bibitem{CH2014}
S.~H. Choi and S.~Y. Ha.
\newblock Global existence of classical solutions to the inelastic
  {V}lasov-{P}oisson-{B}oltzmann system.
\newblock {\em J. Stat. Phys.}, 156(5):948--974, 2014.

\bibitem{ETB2006}
M.~H. Ernst, E.~Trizac, and A.~Barrat.
\newblock The {B}oltzmann equation for driven systems of inelastic soft
  spheres.
\newblock {\em J. Stat. Phys.}, 124(2-4):549--586, 2006.

\bibitem{GambaDiffusively}
I.~M. Gamba, V.~Panferov, and C.~Villani.
\newblock On the {B}oltzmann equation for diffusively excited granular media.
\newblock {\em Comm. Math. Phys.}, 246(3):503--541, 2004.

\bibitem{GradCutoff}
H.~Grad.
\newblock Asymptotic theory of the {B}oltzmann equation. {II}.
\newblock In {\em Rarefied {G}as {D}ynamics ({P}roc. 3rd {I}nternat. {S}ympos.,
  {P}alais de l'{UNESCO}, {P}aris, 1962), {V}ol. {I}}, pages 26--59. Academic
  Press, New York, 1963.

\bibitem{HM19}
J.~Hu and Z.~Ma.
\newblock A fast spectral method for the inelastic {B}oltzmann collision
  operator and application to heated granular gases.
\newblock {\em J. Comput. Phys.}, 385:119--134, 2019.

\bibitem{HQ2020}
J.~Hu and K.~Qi.
\newblock A fast fourier spectral method for the homogeneous boltzmann equation
  with non-cutoff collision kernels.
\newblock {\em J. Comput. Phys.}, 423:109806, 2020.

\bibitem{HQY20}
J.~Hu, K.~Qi, and T.~Yang.
\newblock A new stability and convergence proof of the fourier-galerkin
  spectral method for the spatially homogeneous boltzmann equation.
\newblock {\em preprint, arXiv:2007.05184}, 2020.

\bibitem{LuMouhot2012}
X.~Lu and C.~Mouhot.
\newblock On measure solutions of the {B}oltzmann equation, part {I}: moment
  production and stability estimates.
\newblock {\em J. Differential Equations}, 252(4):3305--3363, 2012.

\bibitem{LuMouhot2015}
X.~Lu and C.~Mouhot.
\newblock On measure solutions of the {B}oltzmann equation, {P}art {II}: {R}ate
  of convergence to equilibrium.
\newblock {\em J. Differential Equations}, 258(11):3742--3810, 2015.

\bibitem{MengLiu2020}
F.~Meng and F.~Liu.
\newblock On the inelastic {B}oltzmann equation for soft potentials with
  diffusion.
\newblock {\em Commun. Pure Appl. Anal.}, 19(11):5197--5217, 2020.

\bibitem{MM2006hardsphere2}
S.~Mischler and C.~Mouhot.
\newblock Cooling process for inelastic {B}oltzmann equations for hard spheres.
  {II}. {S}elf-similar solutions and tail behavior.
\newblock {\em J. Stat. Phys.}, 124(2-4):703--746, 2006.

\bibitem{MM2006hardsphere1}
S.~Mischler, C.~Mouhot, and Rodriguez~R. M.
\newblock Cooling process for inelastic {B}oltzmann equations for hard spheres.
  {I}. {T}he {C}auchy problem.
\newblock {\em J. Stat. Phys.}, 124(2-4):655--702, 2006.

\bibitem{morimoto2012remark}
Y.~Morimoto.
\newblock A remark on {C}annone-{K}arch solutions to the homogeneous
  {B}oltzmann equation for {M}axwellian molecules.
\newblock {\em Kinet. Relat. Models}, 5(3):551--561, 2012.

\bibitem{MWY2015}
Y.~Morimoto, S.~Wang, and T.~Yang.
\newblock A new characterization and global regularity of infinite energy
  solutions to the homogeneous {B}oltzmann equation.
\newblock {\em J. Math. Pures Appl. (9)}, 103(3):809--829, 2015.

\bibitem{morimoto2016measure}
Y.~Morimoto, S.~Wang, and T.~Yang.
\newblock Measure valued solutions to the spatially homogeneous {B}oltzmann
  equation without angular cutoff.
\newblock {\em J. Stat. Phys.}, 165(5):866--906, 2016.

\bibitem{PTfourier1996}
A.~Pulvirenti and G.~Toscani.
\newblock The theory of the nonlinear {B}oltzmann equation for {M}axwell
  molecules in {F}ourier representation.
\newblock {\em Ann. Mat. Pura Appl. (4)}, 171:181--204, 1996.

\bibitem{Qi20}
K.~Qi.
\newblock Measure valued solution to the spatially homogeneous boltzmann
  equation with inelastic long-range interactions.
\newblock {\em preprint, arXiv:2005.08282}, 2020.

\bibitem{ToscaniVillani1999}
G.~Toscani and C.~Villani.
\newblock Probability metrics and uniqueness of the solution to the {B}oltzmann
  equation for a {M}axwell gas.
\newblock {\em J. Stat. Phys.}, 94(3-4):619--637, 1999.

\bibitem{Villani2006granularmaterials}
C.~Villani.
\newblock Mathematics of granular materials.
\newblock {\em J. Stat. Phys.}, 124(2-4):781--822, 2006.

\bibitem{WeiZhang2012}
J.~Wei and X.~Zhang.
\newblock On the {C}auchy problem for the inelastic {B}oltzmann equation with
  external force.
\newblock {\em J. Stat. Phys.}, 146(3):592--609, 2012.

\end{thebibliography}

\end{document}